\documentclass[leqno]{amsart}
\usepackage[utf8]{inputenc}
\usepackage[english]{babel}
\usepackage{newclude}
\usepackage{amsfonts,amssymb,amsmath,amsgen, amsthm}
\usepackage{pdfsync}
\usepackage{enumerate}
\usepackage{centernot}
\usepackage{comment}
\usepackage{hyperref}
\usepackage[a4paper, total={6in, 8in}]{geometry}

\newtheorem{theorem}{Theorem}[section]
\newtheorem{lem}[theorem]{Lemma}
\newtheorem{prop}[theorem]{Proposition}

\theoremstyle{remark}
\newtheorem{remark}[theorem]{Remark}
\theoremstyle{definition}

\def\R{\mathbb R}
\def\C{\mathbb C}
\def\N{\mathbb N}
\def\d{\partial}

\def\T{\mathbb T}
\newcommand{\ii}{\mathrm{i}\mkern1mu}
\def\eps{\varepsilon}
\def\Q{Q_{m,\varepsilon}}
\def\A{\mathbf A}

\title[Cnoidal Waves for the Damped NLS Equation]{Stability of Cnoidal Waves for the Damped Nonlinear Schr\"odinger Equation}

\author[P. Antonelli]{Paolo Antonelli}
\address{Gran Sasso Science Institute, viale Francesco Crispi, 7, 67100 L'Aquila, Italy}
\email{paolo.antonelli@gssi.it}

\author[B. Shakarov]{Boris Shakarov}
\address{Gran Sasso Science Institute, viale Francesco Crispi, 7, 67100 L'Aquila, Italy}
\email{boris.shakarov@gssi.it}

\begin{document}
\maketitle

\begin{abstract}
We consider the cubic nonlinear Schr\"odinger (NLS) equation with a linear damping on the one-dimensional torus and we investigate the stability of some solitary wave profiles within the dissipative dynamics. The undamped cubic NLS equation is well known to admit a family of periodic waves given by Jacobi elliptic functions of cnoidal type.
We show that the family of cnoidal waves is orbitally stable. More precisely, by considering a sufficiently small perturbation of a given cnoidal wave at the initial time, the evolution will always remain close (up to symmetries of the equation) to the cnoidal wave whose mass is modulated according to the dissipative dynamics.
This result extends the concept of orbital stability to this non-Hamiltonian evolution.
\newline
Since cnoidal waves are not exact solutions to the damped NLS, the perturbation is forced away from the family of solitary wave profiles. In order to control this secular growth of the error, we find a first-order approximation of the solitary wave that takes into account the dissipative term. Then we use a suitable, exponentially decreasing Lyapunov functional that controls the $H^1$-norm of the perturbation around the approximated solitons.
\end{abstract}

\section{Introduction}
In this work, we study the following cubic nonlinear Schr\"odinger (NLS) equation with linear damping
\begin{equation}\label{eq:Nls_damp}
\begin{cases}
\ii \partial_t \psi + \partial_{xx} \psi + |\psi|^2 \psi + \ii \varepsilon \psi =0, \\
\psi(0)=\psi_0 \in H^1(\T), 
\end{cases}
\end{equation}
in the $2\pi-$periodic torus $\T$, i.e. $\psi:[0, \infty)\times\T\to\C$. Here $\eps>0$ is a small damping parameter. The cubic NLS equation, namely equation \eqref{eq:Nls_damp} with $\eps=0$, is a canonical model for nonlinear wave propagation \cite{Whit, Ab11, AbSe81}, with applications to nonlinear optics \cite{Ke65, Be98}, plasma physics \cite{Zak, TaWa69} and water waves \cite{Za68}.
\newline
To provide a more accurate description of some physical phenomena, a small dissipative effect can be included in the model. 
In those cases, the classical NLS equation is augmented by viscous terms or linear and nonlinear damping terms \cite{Mal}, see also \cite{PSS}. 
A linear damping, such as in equation \eqref{eq:Nls_damp}, describes a loss term acting at a constant rate. That is the case for instance in plasma physics when the wave slowly releases energy by heating the particles in the plasma \cite{NichGold}. 
Dissipative effects in nonlinear wave propagation described by linearly damped NLS equations appear also in different contexts, see for instance \cite{NB, Segur, PSS, CiPa01}. 
The local well-posedness of solutions to \eqref{eq:Nls_damp} follows from the fact that $H^1(\T)$ is an algebra in one dimension (see for instance Theorem $3.5.1$ in \cite{Ca03}). The global well-posedness follows from the uniform bound on the $L^2$-norm of solutions (that is decreasing in time), uniform bounds on compact time intervals for the energy functional defined in \eqref{eq:enCnIntr} below and the fact that the nonlinearity is mass-subcritical (see the computations in \cite{Ts90} or \cite{In19} for instance). 
\newline
The dissipation in \eqref{eq:Nls_damp} modifies completely the qualitative behavior of solutions. Indeed, equation \eqref{eq:Nls_damp} is no longer Hamiltonian, so the usual functionals associated with physical quantities, such as the total mass
\begin{equation*} 
M[\psi(t)] = \frac{1}{2}\int_0^{2\pi} |\psi(t)|^2dx,
\end{equation*}
and the total energy
\begin{equation}\label{eq:enCnIntr}
E[\psi(t)] = \int_0^{2\pi} \frac{1}{2}|\partial_x \psi(t)|^2- \frac{1}{4} |\psi(t)|^{4} dx,
\end{equation}
are no longer conserved along the dissipative dynamics. 
In particular, for the total mass, we have
\begin{equation}\label{eq:mass}
M[\psi(t)] = e^{-2\varepsilon t} M[\psi_0],
\end{equation}
which implies that all solutions asymptotically converge to zero at an exponential rate.
\newline
Consequently, the presence of damping does not allow solitary wave solutions to equation \eqref{eq:Nls_damp}. Let us recall that the cubic NLS equation admits a large family of periodic waves of the form
\begin{equation*}
\psi(t,x) = e^{\ii \omega t} Q(x),
\end{equation*} 
where $\omega\in\R$ and the profile $Q\in H^1(\T)$ satisfies the equation
\begin{equation} \label{eq:sol_wave}
-\partial_{xx} Q + \omega Q - |Q|^2 Q = 0.
\end{equation}
Non-constant, real-valued, periodic solutions of \eqref{eq:sol_wave} are well known to be given by (scaled versions of) the Jacobi elliptic functions, see Subsection \ref{ss:jac} below for more details.
\newline
This work deals with one particular family of such waves, namely the cnoidal waves. We show that, despite the fact that one specific cnoidal wave cannot be stable under the flow determined by \eqref{eq:Nls_damp}, the whole family actually is. That is, if we denote by $Q_m$ the real-valued $2\pi-$periodic cnoidal wave with total mass $m$, i.e. $M[Q_m]=m$, and by
\newline
\begin{equation}\label{eq:cn_fam}
\mathcal{F} = \{ Q_m, m>0\} 
\end{equation}
the whole family of cnoidal waves, then if the initial datum is sufficiently close to $\mathcal F$, its evolution will always stay close to it. 
More precisely, we are going to show that the solution remains close (up to invariances of the equation) to the cnoidal wave $Q_{m(t)}$ with total mass $m=m(t)$, where according to \eqref{eq:mass}, we set 
\begin{equation}\label{eq:mass2}
m(t)=e^{-2\eps t}m_0,
\end{equation}
with $m_0$ determined by the initial datum.
We will give a more quantitative statement in Theorem \ref{thm:mainCN} below.
\newline
To have a comparison with the Hamiltonian dynamics, let us recall some basic facts about the stability of cnoidal waves for the cubic, undamped NLS equation.
In \cite{GuCoTs17} cnoidal waves are characterized as the unique (up to invariances of the equation) minimizers of the energy with a given mass, in the space of half-anti-periodic functions $H^1(\T)\cap\mathcal A$, see \eqref{eq:hap} below for the precise definition.
This result, combined with the conservation of total energy and mass, readily implies the orbital stability of cnoidal waves against half-anti-periodic perturbations, in the spirit of the argument introduced by Cazenave and Lions \cite{CaLi82}.
Moreover, cnoidal waves are also spectrally stable against periodic perturbations of the same period \cite{IvLa08}. See also \cite{GaPe14, LoMoNaPa21} for related results.
\newline
For the dynamics given by \eqref{eq:Nls_damp}, the presence of the linear damping clearly prevents us to use these arguments so we need to exploit a different strategy.
Let us remark that the modulational stability analysis performed in \cite{We85} also applies to \eqref{eq:Nls_damp}. However, this would provide an orbital stability result for the cnoidal wave $Q_{m_0}$ only up to times of order $\eps^{-1}$.
\newline
Our strategy builds upon the Lyapunov-type approach developed in \cite{We86}. Even though this procedure is extensively used in the context of nonlinear dispersive equations, the dissipative term in \eqref{eq:Nls_damp} introduces several mathematical difficulties that we are going to present in what follows.
\newline
By taking into account the phase shift and space translation invariances of the equation, we first decompose the solution as the sum of the cnoidal wave and a perturbation as follows
\begin{equation*}
\psi(t,x) = (Q_{m(t)}(x - y(t)) + \xi(t,x-y(t)))e^{\ii\gamma(t)}.
\end{equation*}
Notice that the mass $m(t)$ of the cnoidal wave is chosen as in \eqref{eq:mass2}, so to follow the dissipative dynamics, whereas the parameters $\gamma(t)$ and $y(t)$ minimize the distance between $\psi(t, \cdot)$ and the orbit of $Q_{m(t)}$. As for the Hamiltonian case, this choice of parameters eliminates some secular modes in the linearized dynamics.
Consequently, the linearized operator around $Q_{m(t)}$ is positive definite for functions in $H^1(\T)\cap\mathcal A$ satisfying suitable orthogonality conditions. 
However, the coercivity constant is proportional to the total mass (this is related to spectral properties of the linearized operator, see Section \ref{sub:spec}) and hence it degenerates when $m(t)\to0$. 
To overcome this difficulty, we further restrict our analysis to even, half-anti-periodic functions $H^1(\T)\cap\mathcal A^+$, see \eqref{eq:a_p} for the definition of $\mathcal A^+$. 
The even symmetry eliminates the subspace where the degeneracy takes place so that the linearized operator restricted to $H^1(\T)\cap\mathcal A^+$ (with the usual orthogonality conditions) is uniformly coercive in $m>0$. On the other hand, by working on $\mathcal A^+$, we do not need to modulate our solution by a translation parameter anymore.
\newline
The natural candidate Lyapunov functional is given by the difference between the total energy of the solution $\psi$ and the total energy of the cnoidal wave,
\begin{equation}\label{eq:lrough}
\mathcal L[\psi] = E[\psi(t)] - E[Q_{m(t)}] .
\end{equation}
For small perturbations $\xi$, it is possible to prove that $\mathcal L$ is equivalent to the quadratic form associated with the linearized operator around $Q_{m(t)}$, hence it is coercive and it controls the $H^1$-norm of $\xi$.
However, it is not possible to derive satisfactory bounds on $\mathcal L$. This is because, since $Q_m$ does not correspond to any (dissipative) solitary wave profile for \eqref{eq:Nls_damp}, the equation for the perturbation $\xi$ bears a forcing term of size $\eps$, depending on $Q_{m(t)}$, see \eqref{eq:xi} below.
This implies that the smallness of $\xi$ can be inferred only up to times of order $\eps^{-1}$.
\newline
To overcome this difficulty, we determine a first-order correction for $Q_m$ and introduce an approximating profile $Q_{m, \eps}$, that takes into account also the dissipative term. 
The complex-valued profile $Q_{m, \eps}$ retains some of the main properties of $Q_m$. In particular, the linearized operator around $Q_{m, \eps}$ is uniformly coercive. 
Moreover, the perturbation around $Q_{m, \eps}$ satisfies a dynamics where the forcing term is of order $\eps^2$ now. This fact, combined with the smallness of the difference between $Q_m$ and $\Q$,
finally provides the stability property for the family $\mathcal F$ defined in \eqref{eq:cn_fam}.
We remark that a similar approach was also used in \cite{DeGr96} for the damped KdV equation. Our case is more complicated because the solution is complex-valued (hence phase shifts need to be taken into account) and due to the degeneracy of the linearized operator for $m\to 0$, as discussed above.
\newline
We are now ready to state our main result. 
The space $\mathcal A^+$ appearing in the statement represents the space of even, half-anti-periodic functions on $\T$, i.e.
\begin{equation*}
\mathcal A^+=\{u\in L^2(\T)\,:\,u(x+\pi)=-u(x)=-u(-x)\},
\end{equation*}
see also Subsection \ref{ss:per} for a more detailed discussion.

\begin{theorem}\label{thm:mainCN}
Let $\psi_0 \in H^1(\T) \cap \mathcal{A}^+$, $\psi \in C([0,\infty), H^1(\T) \cap \mathcal{A}^+)$ be the corresponding solution and let $m$ be defined by \eqref{eq:mass2}, where $m_0=M[\psi_0]$. There exists $\varepsilon^* = \varepsilon^*(m_0)> 0$ such that if 
\begin{displaymath}
\inf_{\gamma_0 \in [0,2\pi)} \| Q_{m(0)} - e^{\ii \gamma_0}\psi_0\|_{H^1} \leq \varepsilon < \varepsilon^*,
\end{displaymath}
then there exist $\gamma\in C^1(\R;\R)$ and $C = C(m_0) >0$ such that, for any $t >0$, we have
\begin{equation} \label{eq:result}
\| Q_{m(t)} - e^{\ii \gamma(t)}\psi(t)\|_{H^1} \leq C \sqrt{\varepsilon }e^{-\varepsilon t}.
\end{equation}
\end{theorem}

Estimate \eqref{eq:result} not only says that the solution $\psi(t)$ converges to zero as $t\to\infty$, as already remarked in \eqref{eq:mass}, but also that, at any time $t>0$, it remains close to the orbit of $Q_{m(t)}$.
As already discussed before, the main point in our analysis will be to prove estimate \eqref{eq:result} with $Q_m$ replaced by the approximating profile $Q_{m, \eps}$. More precisely, we first need to construct an approximating profile whose equation takes into account the damping term too. Then we define a Lyapunov functional that eventually yields a suitable stability estimate for $Q_{m, \eps}$.
Estimate \eqref{eq:result} then comes from the fact that $Q_{m, \eps}$ is $\eps-$close to $Q_m$.
\newline
While $Q_{m, \eps}$ here is constructed as a first-order correction to $Q_m$, it may also be regarded as an approximate dissipative soliton for the dynamics, in the spirit of \cite{AfAkSo96}. A rigorous construction of a ground state solution for a similar model was done in \cite{HaIbMa21}. We plan to investigate this problem for equation \eqref{eq:Nls_damp} in the next future.
\newline
We conclude the introduction by mentioning some existing results on related models.
The NLS equation with a linear damping term has been widely studied in the Euclidean space, also by considering more general local nonlinearities. In \cite{Ts84}, the formation of singularities in finite time is investigated. In \cite{Ts90}, the same author also studies global solutions, even in a general bounded domain with Dirichlet boundary conditions, see also \cite{OhTo09} for further results in this direction. The asymptotic behavior (decay of solutions to zero) is also studied in \cite{Inui}. In the presence of a forcing term, the asymptotic dynamics is richer and global attractors are studied in \cite{Gh88, WangX, GM}.
The NLS equation was also studied with a nonlinear damping term, see \cite{AS, AnCaSp15, Da14}.
\newline
This work is organized as follows: in Section \ref{sec:prelCN} we will set our definitions, notations and recall preliminary results. In Section \ref{sec:first}, we will show that using the Lyapunov functional \eqref{eq:lrough} to control the $H^1$-norm of the difference $e^{\ii \gamma(t)}\psi(t) - Q_{m(t)}$ without perturbing $Q_m$ implies \eqref{eq:result} only if the initial mass $M[\psi_0]$ is small enough. In Section \ref{sec:Adj}, we will refine the profile $Q_m$ incorporating the first-order effect of the damping, obtaining a family of perturbed solitary waves $Q_{m,\varepsilon}$. We will then control the difference $\psi(t) - Q_{m(t),\varepsilon}$ with a modified Lyapunov function and prove that this functional provides a satisfactory bound.

\section{Preliminaries} \label{sec:prelCN}
 For any given functional $S:H^1(\T) \to \R$, we adopt the notation $S'[v]$
for the Gateaux derivative of $S$, i.e.
\begin{equation*}
S'[v](\xi)=\lim_{h\to0}\frac{S[v+h\xi]-S[v]}{h}.
\end{equation*}
In the same way, we will denote by $S''[v]$ the second variation of $S$ computed in $v$ 
\begin{equation*}
 (S''[v]\xi,\xi)=\frac{d^2}{dh^2} S[v+h\xi]_{|h=0}.
\end{equation*}

\subsection{Periodic spaces}\label{ss:per}
We define the following spaces
	\begin{equation*}
		L^2(\T) = \{u \in L^2_{loc}(\R); \, \forall x\in \R,\, u(x+ 2\pi) = u(x) \},
	\end{equation*} 
 \begin{equation*}
		H^1(\T) = \{u \in H^1_{loc}(\R); \, \forall x\in \R,\, u(x+ 2\pi) = u(x) \},
	\end{equation*} 
where in the second definition we have slightly abused notations and the precise definition of Lebesgue spaces on the periodic torus can be found for example in \cite[p. 238]{Fo99}. \newline
The inner product in $L^2(\T)$ is denoted by
\begin{equation*} 
	 (f,g) = Re\int_0^{2\pi} f \bar{g} dx.
\end{equation*}
With the same remark as above, we also define 
\begin{equation}\label{eq:hap}
	\mathcal{A} = \left\{ u \in L^2(\T) ; \, u(x + \pi) = -u(x) \right\}, \quad \mathcal{S} = \left\{ u \in L^2(\T) ; \, u(x + \pi) = u(x) \right\}.
\end{equation} 
We refer to the set $\mathcal A$ as the set of half-anti-periodic functions, whereas $\mathcal S$ is the set of periodic functions with period $\pi$. 
Thus $L^2(\T)$ can be decomposed as the direct sum 
	\begin{equation*}
		L^2(\T) = \mathcal{S} \oplus \mathcal{A}.
	\end{equation*}
	Indeed, any $f\in L^2(\T)$ can be written as
	\begin{displaymath}
		f(x) = \frac{1}{2} (f(x) + f(x + \pi)) + \frac{1}{2} (f(x) - f(x + \pi)),
	\end{displaymath} 
	where
	\begin{equation*}
	 f(x) + f(x + \pi) \in \mathcal S, \quad f(x) - f(x + \pi) \in \mathcal A.
	\end{equation*}
	We further decompose $\mathcal{A}$ into the following subsets
	\begin{equation*}
		\mathcal{A} = \mathcal{A}^+ \oplus \mathcal{A}^-,
	\end{equation*}
	where
$\mathcal A^+$ (resp. $\mathcal A^-$) denotes the set of even (resp. odd) functions in $\mathcal A$, i.e.
	\begin{equation}\label{eq:a_p}
		\mathcal{A}^\pm = \left\{ u \in \mathcal{A} ; \, u(-x) = \pm u(x) \right\}.
	\end{equation}	
	Indeed, for any $f \in \mathcal{A}$, we can write
	\begin{equation*}
		f(x) = \frac{1}{2} (f(x) + f(-x)) + \frac{1}{2} (f(x) - f(-x)),
	\end{equation*}
	where 
	\begin{equation*}
	 f(x) + f(-x) \in \mathcal A^+, \quad f(x) - f(-x) \in \mathcal A^-.
	\end{equation*}
	Thus, we obtain the decomposition 
	\begin{equation*}
		L^2(\T) = \mathcal{S} \oplus \mathcal{A}^+ \oplus \mathcal{A}^-,
	\end{equation*}
	and each of these subspaces is invariant under equation \eqref{eq:Nls_damp}, since 
	\begin{equation*}
		\psi \in \mathcal{S}\ (\mbox{resp. } \mathcal{A}^\pm) \rightarrow \partial_{xx} \psi + |\psi|^2\psi +\ii \varepsilon \psi \in \mathcal{S} \ (\mbox{resp. } \mathcal{A}^\pm).
	\end{equation*} 
	In particular, if $\psi_0 \in H^1(\T) \cap \mathcal{A}^+$, then the corresponding solution is also contained in the same space, namely $\psi \in C([0,\infty), H^1(\T) \cap \mathcal{A}^+)$. 
 Throughout this work, we will use the following notation
 \begin{equation}\label{eq:defA}
 \mathbf A = H^1(\T) \cap \mathcal{A}^+.
 \end{equation}

\subsection{Jacobi elliptic functions}\label{ss:jac}
We introduce the Jacobi elliptic functions. The reader might refer to treatises on elliptic functions (e.g., \cite{ByFr71, La13}) for more details. 
\newline
Given $k\in(0, 1)$, we define the incomplete elliptic integral of the first kind as
\begin{equation*}
	F(\phi,k) = \int_0^\phi \frac{d\theta}{\sqrt{1 - k^2\sin^2 (\theta)}}.
\end{equation*}
The Jacobi elliptic function of cnoidal, snoidal and dnoidal types are defined through the inverse of $F(\cdot,k)$ so that given $x\in \R$ and $k \in [0,1]$, by denoting $\phi = F^{-1} (x,k)$, we define
\begin{equation*} 
	cn(x,k) = \cos(\phi), \quad sn(x,k) = \sin(\phi), \quad dn(x,k) = \sqrt{1 - k^2 \sin^2(\phi)}.
\end{equation*}
The complete elliptic integral of the first kind is defined by
\begin{equation*}
	K(k) = F\left( \frac{\pi}{2}, k \right) = \int_0^\frac{\pi}{2} \frac{d\theta}{ \sqrt{ 1 - k^2 \sin^2 \theta}},
\end{equation*}
and it satisfies
\begin{equation}\label{eq:K_as}
 \lim_{k\to 0} K(k) = \frac{\pi}{2}, \quad \lim_{k \to 1} K(k) = +\infty.
\end{equation}
The functions $cn$ and $sn$ are $4 K(k)$-periodic, while $dn$ is $2K(k)$-periodic. 
This implies that, by scaling, the functions $cn\left(\frac{2K(k)}{\pi}x;k\right)$ and 
$sn\left(\frac{2K(k)}{\pi}x;k\right)$ are $2\pi-$periodic, while 
$dn\left(\frac{2K(k)}{\pi}x;k\right)$ is $\pi-$periodic and in particular
\begin{equation*}
 dn\left(\frac{2K(k)}{\pi}\cdot,k\right)\in\mathcal S,\quad
 sn\left(\frac{2K(k)}{\pi}\cdot,k\right)\in\mathcal A^-,\quad
 cn\left(\frac{2K(k)}{\pi}\cdot,k\right)\in\mathcal A^+.
\end{equation*}
It is straightforward to check by direct computations that, after a suitable scaling, the functions $sn, cn$ and $dn$ satisfy equation \eqref{eq:sol_wave} with a suitable $\omega\in\R$. Since we mainly focus on the space $\A$, in what follows we consider solitary waves of cnoidal type. The profile $cn(x,k)$ satisfies the equation
\begin{equation*}
 \partial_{xx} Q + (1 - 2k^2) Q + 2k^2 |Q|^2Q = 0,
\end{equation*}
and thus the scaled function
\begin{equation}\label{eq:Qm}
	Q_k(x) = \sqrt{2} k \frac{2 K(k)}{\pi} cn \left(\frac{2 K(k)}{\pi} x ,k \right),
\end{equation}	
satisfies \eqref{eq:sol_wave}, with $\omega=\omega(k)$ given by
\begin{equation}\label{eq:omegaCN}
 \omega(k) = \frac{4 K^2(k) (2k^2 - 1)	}{\pi^2}.
\end{equation}
A direct computation shows that it is also possible to express the total mass of $Q_k$ as a function of $k$, namely
\begin{equation*}
 M[Q_k]= \frac{8}{\pi} K(k) \left( \int_0^\frac{\pi}{2} \sqrt{1 - k^2 \sin^2 \theta} d\theta - (1 - k^2) K(k) \right).
\end{equation*}
By exploiting the properties of Jacobi elliptic functions, it may be checked that the function
\begin{equation}\label{eq:k}
 k\mapsto m(k) = \frac{8}{\pi} k^2 K(k) \int_0^\frac{\pi}{2} \frac{\cos^2(\theta)}{\sqrt{1 - k^2\sin(\theta)}}d\theta
\end{equation}
is bijective and strictly increasing, hence invertible, see \cite{GuCoTs17} and also the Proposition below. 
For this reason, when there is no source of ambiguity, we adopt also the notation $Q_m$ to denote $Q_{k(m)}$ as defined in \eqref{eq:Qm}. Analogously, we write $\omega(m)$ to denote $\omega(k(m))$, with $\omega$ defined in \eqref{eq:omegaCN}. The next Proposition collects some properties of cnoidal waves that will be useful in our analysis, for a proof of these results, we refer to \cite[Proposition $3.4$]{GuCoTs17}.

\begin{prop} \label{thm:cnoidal}
 For any $m>0$, the function $Q_m$ defined in \eqref{eq:Qm} is the unique minimizer of the problem	
 \begin{equation}\label{eq:min_prob} 
 \mu = 	\inf \left\{ E(u): \, u \in H^1(\T) \cap \mathcal{A}, M[u] = m \right\},
\end{equation}
up to translations and phase shifts, where $k = k(m)\in(0,1)$ is determined by inverting the bijective and strictly increasing function in \eqref{eq:k}.
\end{prop}

Let us also remark that by using \eqref{eq:K_as} and \eqref{eq:omegaCN}, we have
\begin{equation*}
 \lim_{k \to 0} \omega(k) =\lim_{m\to0}\omega(k(m))= -1.
\end{equation*}

\begin{remark}
Since the profile $Q_m$ is the global unique minimizer (up to invariances) of the convex functional $E$, by combining the classical argument by Cazenave and Lions \cite{CaLi82} and Proposition \ref{thm:cnoidal}, it is possible to show that $Q_m$ is orbitally stable in $H^1(\T) \cap \mathcal{A}$. This was already observed in \cite{GuCoTs17}. The orbital stability was proved also in \cite{GaHa07,Pa07} by an adaptation of the approach introduced in \cite{GrShSt87}.
\end{remark}

\subsection{Linearized Operator}\label{sub:spec}
In this subsection, we outline some preliminary results about the linearized operator around $Q_k$ and we refer to \cite[Section 4]{GuCoTs17} for details. Decomposing the solution as $ \psi(t,x) = e^{\ii \omega t}( Q_k + \xi)$ leads to the equation for $\xi$
\begin{equation*}
 \ii \partial_t \xi = L(\xi) - N(\xi) - \ii D(\xi)
\end{equation*}
where 
\begin{equation}\label{eq:lin_oper}
 L(\xi) = -\partial_{xx} \xi + \omega(k) \xi - Q_k^2\xi - 2Q^2_k Re(\xi)
 \end{equation}
 is the linearization around $Q_k$, the nonlinear terms are
 \begin{equation*}
 N(\xi) = |\xi|^2 Q_k + 2Re(\xi) Q_k \xi + |\xi|^2 \xi,
 \end{equation*}
 and the damping term is given by
 \begin{equation*}
 D(\xi) = \varepsilon(Q_k + \xi).
 \end{equation*}
We separate the real and imaginary parts of the linearization term as 
\begin{equation}\label{eq:lin_oper2}
 L(\xi) = L_{+}^{(k)} Re(\xi) + \ii L_{-}^{(k)} Im(\xi)
\end{equation}
where we define $ L_{\pm}^{(k)}:H^1(\T) \to H^{-1}(\T)$ as 
\begin{equation}\label{eq:lin_opCN}
	 \begin{split}
			&L_{+}^{(k)} = - \partial_{xx} + \omega(k) - 3 Q_k^2 , \\
			&L_-^{(k)} = - \partial_{xx} + \omega(k) - Q_k^2.
 	\end{split}
 \end{equation}
By \eqref{eq:Qm}, \eqref{eq:omegaCN} and \eqref{eq:lin_opCN}, we get 
	\begin{equation*}
		\begin{split}
			&L_+^{(k)} = - \partial_{xx} - \frac{4K^2(k)}{\pi^2} \left[(1 - 2k^2) + 6k^2 cn^2 ( 2 K(k)x/\pi,k) \right], \\
			&L_-^{(k)} = - \partial_{xx} - \frac{4K^2(k)}{\pi^2} \left[(1 - 2k^2) + 2k^2 cn^2 ( 2 K(k)x/\pi,k) \right].
		\end{split}
	\end{equation*}
 The operators $L_\pm^{(k)}$ enter in the framework of Schr\"odinger operators with periodic potentials whose spectral properties have been extensively studied (see e.g. \cite{Ea73}). In particular, we recall that given an operator $L = -\partial_{xx} + V$ with $2\pi$-periodic potential $V$, the eigenvalues of $L$ on $L^2(\T)$ satisfy 
 \begin{equation}\label{eq:eigOrder}
 \lambda_0 < \lambda_1 \leq \lambda_2 < \lambda_3 \leq \lambda_4< ... .
 \end{equation}
 with corresponding eigenfunctions $\varphi_n$ such that $\varphi_0$ has no zeros, $\varphi_{2m + 1}$ and $\varphi_{2m + 2}$ have exactly $2m + 2$ zeros in $[0,2\pi)$, (\cite[p. 39]{Ea73}). We now analyze the spectrum of $L_{\pm}$ more in detail. 
 
{\bf Spectrum of $L_-$:} Let $sp(L_-^{(k)}) = \{ \lambda_n(k) , \, n \in \N \} \subset \R$, with $\lambda_{j} \leq \lambda_{j+1}$. The first three eigenvalues are given by
	\begin{equation}\label{eq:L-sp}
		\lambda_0(k) = (k^2 - 1) \frac{4 K^2(k)}{\pi^2}, \quad \lambda_1(k) = 0, \quad \lambda_2(k) = k^2 \frac{4 K^2(k)}{\pi^2}.
	\end{equation}
The fourth eigenvalue satisfies the inequality
 \begin{equation}\label{eq:lambda3}
 \inf_{k\in[0,1)} \lambda_3 (k) > 0.
 \end{equation}
Indeed, from \eqref{eq:eigOrder} we have $\lambda_3(k) >\lambda_2(k) > 0$ for any $k \in (0,1)$ and moreover $L_-^{(k)} \to L_-^{(0)} = - \partial_{xx} + 1$ as $k \to 0$ which implies $\lambda_3(k) \to 1$. The corresponding eigenfunctions are given by
\begin{equation*}
\begin{aligned}
 &\phi_0=dn\left(\frac{2K(k)}{\pi}\cdot,k\right)\in\mathcal S, \quad
 \phi_1=cn\left(\frac{2K(k)}{\pi}\cdot,k\right)\in\mathcal A^+,\\
 &\phi_2=sn\left(\frac{2K(k)}{\pi}\cdot,k\right)\in\mathcal A^-.
\end{aligned}
\end{equation*}
	Since $\lambda_0(k) < 0$, the operator $L_-^{(k)}$ is not positive definite. Indeed, the cnoidal wave \eqref{eq:Qm} is not a global minimizer of problem \eqref{eq:min_prob} in the whole space $H^1(\T)$. The global minimizer is given by the dnoidal profile, see \cite[Proposition $3.2$]{GuCoTs17} for instance. 
 Nonetheless, the restriction of the operator $L_-^{(k)}$ to the subspace $\mathcal{A}$ is positive semi-definite, since the eigenfunction $\phi_0\notin \mathcal{A}$. 
 So, for any $f \in \mathcal{A} \cap H^1(\T)$ orthogonal to $\phi_1$, we have
	\begin{equation*}
		(L_- f,f) \geq \lambda_2(k) \| f\|_{L^2}^2.
	\end{equation*} 
However, $\lambda_2(k) \to 0$ as $k \to 0$, or equivalently $\lambda_2(m) \to 0$ as $m \to 0$, see \eqref{eq:k}. To deal with this degeneracy, we restrict the analysis to $\A$ and obtain that $L_-$ is uniformly coercive, namely
	\begin{equation*}
		(L_-^{(k)} f,f) \geq \lambda_3(k) \| f\|_{L^2}^2 \geq \inf_{k\in [0,1)} \lambda_3(k) \| f\|_{L^2}^2.
	\end{equation*}

	{\bf Spectrum of $L_+$:}
	Let $sp(L_+^{(k)}) = \{ \mu_n(k) , \, n \in \N \} \subset \R$, with $\mu_{j} \leq \mu_{j+1}$. The first five eigenvalues are given by
	\begin{equation}\label{eq:L+sp}
		\begin{split}
			&	\mu_0(k) = (c_-(k) - 3k^2) \frac{4 K^2(k)}{\pi^2}, \quad \mu_1(k) = -3k^2 \frac{4 K^2(k)}{\pi^2}, \quad \mu_2(k)= 0, \\
			& \mu_3(k)= (3 - 3k^2) \frac{4 K^2(k)}{\pi^2}, \quad \mu_4(k) = (c_+(k) - 3k^2) \frac{4 K^2(k)}{\pi^2},
		\end{split}
	\end{equation}
	where 
	\begin{equation*}
		 c_\pm(k) = k^2 + 1 \pm 2 \sqrt{k^4 - k^2 + 1}.
	\end{equation*}
The corresponding eigenfunctions are such that 
	\begin{equation}\label{eq:L+sp1}
 \begin{aligned}
 &u_0, u_3, u_4 \in \mathcal{S}, \\
 &u_1 = \partial_x sn( 2 K(k)x/\pi,k) \in \mathcal{A}^+, \quad u_2 = \partial_x cn( 2 K(k)x/\pi,k) \in \mathcal{A}^-. 
 \end{aligned} 
	\end{equation}
 This implies that when restricting $L_+^{(k)}$ to the subspace $\mathcal A$, the first three eigenvalues are $\mu_1(k)<0$, $\mu_2(k)=0$, $\mu_5(k)>0$ and there exists a constant $c>0$ such that $\mu_5(k) \geq c$ for any $k \in [0,1]$. Consequently, for every $f \in \mathbf A$ orthogonal to $u_1$, we have the uniform coercivity
	\begin{equation*}
		(L_+ f,f) \geq \mu_5(k) \| f\|_{L^2}^2 \geq c\| f\|_{L^2}^2.
	\end{equation*}

\section{Lyapunov control around the modulated solitary wave}\label{sec:first}

In this section, we start our analysis of the linearized dynamics around the periodic wave $e^{i\omega t}Q_m$, with $Q_m$, $\omega=\omega(m)$ given by \eqref{eq:Qm}, \eqref{eq:omegaCN}, respectively.
As discussed in Proposition \ref{thm:cnoidal} and in the previous section, the parameter $m$ denotes the total mass of the periodic wave. For this reason, we let $m$ depend on time, to follow the dissipative dynamics \eqref{eq:Nls_damp}, according to \eqref{eq:mass}. \newline
To obtain a control of the perturbation around the periodic wave, we first need to infer a suitable coercivity property of the linearized operator $L$ defined in \eqref{eq:lin_oper}, \eqref{eq:lin_oper2} and \eqref{eq:lin_opCN}. 
One of the main points in our analysis relies on the fact that the coercivity constant needs to be uniform in $m\in[0, \infty)$. 
As already outlined in Subsection \ref{sub:spec}, this condition compels us to consider only even, half-anti-periodic perturbations. The coercivity of the linearized operator then yields an $H^1$-control of the perturbation, provided we infer suitable bounds on the Lyapunov functional. This is achieved in Proposition \ref{thm:L_weak}. However, we will see that the bound proved there is not sufficient to show the stability of the family of cnoidal waves and it will lead us to improve our perturbative analysis in the next section. \newline
Let us now consider $\psi_0 \in \A$, where $\A$ is defined in \eqref{eq:defA}, such that $M[\psi_0]=m_0$, let $\psi\in C([0, \infty);\A)$ be the corresponding solution to \eqref{eq:Nls_damp}.
We define the function
\begin{equation}\label{eq:mass_diss}
m(t) = M[\psi(t)]=m_0e^{-2\eps t},
\end{equation}
where the last identity follows from \eqref{eq:mass}.
By taking into account the phase shift invariance, for any $t\geq0$, we define the distance between $\psi$ and the family $\mathcal F$ of cnoidal waves defined in \eqref{eq:cn_fam}, as
\begin{equation}\label{eq:defdis}
d_{\mathcal F}(\psi(t))=\inf_{\gamma\in[0, 2\pi]}\|e^{-\ii \gamma}\psi(t)-Q_{m(t)}\|_{L^2}.
\end{equation}
Let us recall that the parameter $m(t)$ is chosen in such a way that $\|\psi(t)\|_{L^2}=\|Q_{m(t)}\|_{L^2}$. Moreover, in \eqref{eq:defdis} we do not take into account space translations, as our analysis is restricted to $\A$. 
It is straightforward to see that for any $t\geq0$ the infimum in \eqref{eq:defdis} is attained, namely there exists $\gamma \in C^1(\R,\R)$ such that
\begin{equation}\label{eq:gamma}
 d_{\mathcal F}(\psi(t))=\|e^{-\ii \gamma(t)}\psi(t)-Q_{m(t)}\|_{L^2}.
\end{equation}
The perturbation $\xi\in C([0, \infty);\A)$ around the periodic wave is defined by
\begin{equation}\label{eq:decomCN}
 \xi(t, x)	= e^{-\ii \gamma(t)}\psi(t,x)- Q_{m(t)}(x),
\end{equation}
or equivalently, we may write
\begin{equation}\label{eq:decopsi}
 \psi(t, x)=e^{\ii\gamma(t)}\left(Q_{m(t)}(x)+\xi(t, x)\right).
\end{equation}
Definitions \eqref{eq:defdis}, \eqref{eq:gamma} and \eqref{eq:decomCN} also imply that
\begin{equation*}
d_{\mathcal F}(\psi(t))=\|\xi(t)\|_{L^2}.
\end{equation*}
The following properties for $\xi$ come from the previous definitions.

\begin{prop}\label{prop:ort}
 Let $\psi_0 \in \A$, $\psi\in C([0,\infty), \A) $ be its corresponding evolution, $m$ and $\gamma$ be defined as in \eqref{eq:mass} and \eqref{eq:gamma}, respectively. Then for any $t\geq0$, we have 
 \begin{equation}\label{eq:ortho_con}
		(\xi(t, \cdot),\ii Q_{m(t)}) = 0, 
 \end{equation}
 \begin{equation}\label{eq:ortho_con_2}
 (\xi(t, \cdot), Q_{m(t)}) = - M[\xi(t)].
 \end{equation}
\end{prop}

\begin{proof}
 Condition \eqref{eq:ortho_con_2} follows from $\| Q_m\|_{L^2} = \| \psi \|_{L^2}$ and decomposition \eqref{eq:decomCN} which implies that 
	\begin{equation*}
		\| \psi\|_{L^2}^2 = \| \xi\|_{L^2}^2 + \| Q_m\|_{L^2}^2 + 2(\xi,Q_m) = \| Q_m\|_{L^2}^2.
		\end{equation*}
		Moreover, \eqref{eq:gamma} and \eqref{eq:decopsi} imply that for any $t > 0$ we have
		\begin{equation*}
		\begin{split}
		 0 = - \frac{1}{2}\frac{d}{d\gamma}\int | e^{-\ii \gamma}\psi - Q_m|^2 dx = Re\int \ii e^{-i\gamma} \psi Q_m dx = Re\int \ii (Q_m + \xi)Q_m dx = (\xi, \ii Q_m)
		 \end{split}
	\end{equation*}
	which implies \eqref{eq:ortho_con}.
\end{proof}	

By using the Cauchy-Schwarz inequality, \eqref{eq:ortho_con_2} yields the following bound on the perturbation,
\begin{equation}\label{eq:bnd1}
\| \xi(t)\|_{L^2} \leq 2\| Q_{m(t)}\|_{L^2}
=2\|\psi_0\|_{L^2}e^{-\eps t},
\end{equation}
where the last identity comes from \eqref{eq:mass_diss}.
Consequently, the $L^2$-norm of $\xi$ remains small for all times, provided that the initial mass is sufficiently small.
Our main result amounts to improving \eqref{eq:bnd1} to the following bound
\begin{equation*}
\| \xi(t)\|_{H^1} \lesssim \varepsilon e^{-\varepsilon t},
\end{equation*}
which implies a uniform control on the smallness of $\xi$, regardless of the size of the initial datum $\psi_0$.
This is achieved by exploiting a Lyapunov functional approach, as developed by Weinstein \cite{We86}. Taking into account that $Q_m$ is a minimizer of the problem \eqref{eq:min_prob}, the natural choice for our Lyapunov functional is given by 
	\begin{equation*}
		\mathcal{L}[\psi(t)] = E[\psi(t)] - E[Q_{m(t)}].
	\end{equation*}
 By recalling that $\psi(t)$ and $Q_{m(t)}$ have the same mass, we also write equivalently
 \begin{equation}\label{eq:Lyap}
		\mathcal{L}[\psi(t)] = S_{m(t)}[\psi(t)] - S_{m(t)}[Q_{m(t)}],
	\end{equation}
 where $S_m$ is the action functional, defined by
	\begin{equation}\label{eq:actionS}
		S_{m}[\psi] = E[\psi] + \omega(m) M[\psi].
	\end{equation}
	Let us plug \eqref{eq:decomCN} into \eqref{eq:Lyap} and use \eqref{eq:sol_wave}, \eqref{eq:ortho_con_2} to obtain
	\begin{equation}\label{eq:L_Lyapunov}
		\begin{split}
			\mathcal{L}[\psi] & = \int_0^{2\pi} \frac{1}{2} |\partial_x \xi|^2 - [\partial_{xx} Q_m + Q_m^3] \xi_R dx - \int_0^{2\pi} \frac{3}{2}Q_m^2 \xi_R^2 + \frac{1}{2} Q_m^2 \xi_I^2dx \\
			& - \int_0^{2\pi} Q_m \xi_R^3 + Q_m \xi_I^2 \xi_R + \frac{1}{4} |\xi|^4 dx \\
			& = \frac{1}{2} \left[(L_+\xi_R,\xi_R) + (L_- \xi_I, \xi_I) \right] - \int_0^{2\pi} Q_m \xi_R^3 + Q_m \xi_I^2 \xi_R + \frac{1}{4} |\xi|^4 dx.
		\end{split} 
	\end{equation}
Here $\xi_R = Re(\xi)$, $\xi_I = Im(\xi)$, $L_+ = L_+^{(k)}, L_- = L_-^{(k)}$ are the real and imaginary part of the linearized operators around $Q_{m(t)}$ introduced in \eqref{eq:lin_opCN} and $k$ is defined from the inverse of formula \eqref{eq:k}.	
The quadratic part appearing on the right-hand side of \eqref{eq:L_Lyapunov} may be identified as
\begin{equation}\label{eq:quad}
(S_m''[Q_m]\xi,\xi) =	(L_+ \xi_R, \xi_R) + (L_- \xi_I, \xi_I),
\end{equation} 
where $S_m''[Q_m]$ is the second variation of the action $S_{m}$ defined in \eqref{eq:actionS},
\begin{equation}\label{eq:secvar}
 S_m''[Q_m] \xi = - \partial_{xx} \xi + \omega \xi - Q_m^2\xi -2 Re(Q_m \bar{\xi})Q_m. 
\end{equation}
Our next goal is to show the coercivity of the functional $\mathcal L[\psi]$. The next proposition states that its quadratic part is coercive.

\begin{prop} \label{thm:pos}
 There exists a constant $C>0$ such that for any $m \in [0, m_0]$ and
	\begin{equation}\label{eq:X_sp}
	\xi \in	\mathcal{X}= \left\{ u \in \A: \, (u,Q_m) = (u,\ii Q_m) = 0 \right\}
	\end{equation}
	we have
		\begin{equation} \label{eq:H1_pos}
			(L_+ \xi_R, \xi_R) + (L_- \xi_I, \xi_I) \geq C\| \xi\|_{H^1}^2.
		\end{equation} 
\end{prop}

\begin{proof} 
By the spectral analysis presented in Subsection \ref{sub:spec}, we infer that
\begin{equation}\label{eq:L_min_coerc}
(L_- f, f) \geq C_-\| f\|_{L^2}^2,
\end{equation}
where $C_->0$ is given by
\begin{equation*}
C_- = \inf_{k\in [0,1)} \lambda_3(k) >0,
\end{equation*}
and $\lambda_3(k)$ is the fourth eigenvalue in $sp(L_-^{(k)})$, see \eqref{eq:lambda3}. \newline
We need to prove a similar coercivity property for $L_+$. We first claim that 
		\begin{equation}\label{eq:tau_0}
			\tau_0 = \inf \left\{ ( L_+ v,v )\,; \, v \in H^1(\T) \cap \mathcal{A}, \ (v,Q_m) = 0, \| v\|_{L^2} = 1\right\} = 0.
		\end{equation}
		Observe that $\tau_0 > -\infty$ because of the the lower bound 
		\begin{equation*}
			(L_+ v,v) \geq \omega\| v\|_{L^2}^2 - 3(Q_m^2 v, v) \geq (\omega - 3 \| Q_m\|_{L^\infty}^2) \|v\|_{L^2}^2.
		\end{equation*}
Moreover, notice that $\partial_x Q_m \in H^1(\T) \cap \mathcal{A}$ (in fact, from \eqref{eq:L+sp1}, we have $\partial_xQ_m\in\mathcal A^-$), $(\partial_x Q_m, Q_m) = 0$ and 
$L_+ \partial_x Q_m= 0$. 
This implies that $\tau_0 \leq 0$. 
Let us now define $\chi = -\partial_\omega Q_m$. By differentiating equation \eqref{eq:sol_wave} with respect to $\omega$ we have
\begin{equation}\label{eq:chi}
L_+ \chi = Q_m.
\end{equation}
Moreover, it is known from \cite[equation $4.11$]{Pa07} that
\begin{equation*} 
(L_+ \chi, \chi) = - \left( Q_m, \partial_\omega Q_m \right) = - \frac{1}{2} \frac{d}{d\omega} \|Q_m\|_{L^2}^2 < 0.
\end{equation*}
 We decompose $\chi$ as the sum
		\begin{equation*} 
			\chi = \alpha u_1+ \beta\partial_x Q_m + p
		\end{equation*}
where $u_1\in\mathcal A$ is the normalized eigenfunction of $L_+$ relative to $\mu_1$, see \eqref{eq:L+sp} and \eqref{eq:L+sp1}, $\alpha, \beta \in \R$ and $p \in \mathcal A_0 = \{ v \in H^1(\T) \cap \mathcal{A}; \ (v , \partial_x Q_m) = (v, \phi_0) = 0 \}$. One can easily see that 
	\begin{equation*}
			\begin{split}
				(L_+ p,p) = (L_+\chi, \chi) - \mu_1 \alpha^2 < -\mu_1 \alpha^2 .
			\end{split}
 \end{equation*}
Let us now consider an arbitrary $v \in H^1(\T) \cap \mathcal{A}$, satisfying $(v, Q_m) = 0$ and $\| v\|_{L^2} = 1$. As before, we decompose $v$ as 
		\begin{equation*} 
			v = a u_1 + b \partial_x Q_m + q,
		\end{equation*}
		with $a, b \in \R$ and $q \in \mathcal A_0$.
By using \eqref{eq:chi}, we have that
	\begin{equation*}
			\begin{split}
				&0 = (Q_m,v) =(L_+ \chi, v) = (\mu_1 \alpha u_1 + L_+ p, a u_1 + b \partial_x Q_m + q) = \mu_1 a \alpha+ (L_+ p, q).
			\end{split}
		\end{equation*}	
		Since $L_+$ restricted to $\mathcal A_0$ is positive-definite, we can use the Cauchy-Schwarz inequality 
		\begin{equation*} 
			|(L_+ f,g) |^2 \leq (L_+ f,f)(L_+g,g),
		\end{equation*}
		for every $f,g \in \mathcal{A}_0$. Thus we get
		\begin{equation*}
			\begin{split}
				(L_+ v,v) = \mu_1 a^2 + (L_+ q,q) \geq \mu_1 a^2 + \frac{|(L_+ p,q) |^2}{(L_+ p,p)} > \mu_1 a^2 - \frac{\alpha^2 a^2 \mu_1^2}{\alpha^2 \mu_1} = 0,
			\end{split}
		\end{equation*}
which implies that $\tau_0 = 0$. 
		The next step is to prove that 
		\begin{equation} \label{eq:tau1}
			\tau = \inf \left\{ ( L_+ v,v )\,;\, v \in H^1(\T) \cap \mathcal{A}, \ (v,Q_m) = (v, \partial_x Q_m) = 0, \| v\|_{L^2} = 1\right\} > 0. 
		\end{equation}
We straightforwardly have that $\tau\geq\tau_0=0$. \newline
We show that the infimum in \eqref{eq:tau1} is achieved. Let $\{ u_j\}_{j \in \N} \subset H^1(\T) \cap \mathcal{A} $ be a minimizing sequence such that 
\begin{equation}\label{eq:minse}
			( L_+ u_j,u_j ) \rightarrow \tau, \quad \| u_j\|_{L^2} = 1, \quad (u_j,Q_m) = (u_j, \partial_x Q_m) = 0.
\end{equation} 
From the first condition in \eqref{eq:minse} it follows that $(L_+ u_j, u_j)$ is uniformly bounded from above. This implies that
\begin{equation*} 
			\|\partial_x u_j\|_{L^2}^2 = (L_+ u_j,u_j) - \omega + 3 (Q_m^2, u_j^2) \lesssim 1
\end{equation*}
for all $j \in \N$. So $\{ u_j\}_{j \in \N}$ is a bounded sequence in $H^1(\T) \cap \mathcal{A}$. There exists a subsequence, still denoted by $\{ u_j\}_{j \in \N}$ and $u \in H^1(\T) \cap \mathcal{A} $ such that $u_j \rightharpoonup u \in H^1(\T) \cap \mathcal{A}$. In particular, $(u,Q_m) = 0$ and $(u, \partial_x Q_m) = 0$. 
From the compact embedding $H^1(\T) \cap \mathcal{A} \hookrightarrow L^2(\T)$, we have $\| u\|_{L^2} = 1.$	By the lower semicontinuity, it follows that
\begin{equation*}
			\tau \leq (L_+ u,u) \leq \liminf_{j \in \N} (L_+ u_j,u_j) = \tau.
\end{equation*}
Therefore the minimum is achieved by $u \in H^1(\T) \cap \mathcal{A}$, with 
$\| u\|_{L^2} = 1$, $(u,Q_m) = (u, \partial_x Q_m) = 0$.
Next, let us assume by contradiction that $\tau = 0$. Then there exist some Lagrange multipliers $\alpha, \beta, \gamma \in \R$ such that
		\begin{equation}\label{eq:lu} 
			L_+ u = \alpha u + \beta Q_m + \gamma \partial_x Q_m.
		\end{equation}
		Multiplying \eqref{eq:lu} by $u$ and integrating in space yields
		\begin{equation*}
		 0 = \tau = (L_+ u,u) = \alpha \|u\|_{L^2}^2,
		\end{equation*}
and thus $\alpha = 0.$ Moreover, by taking the scalar product of equation \eqref{eq:lu} with $\partial_x Q_m$, we get
		\begin{equation*} 
			0 = (u, L_+ \partial_x Q_m) = \gamma \|\partial_x Q_m\|_{L^2}^2 
		\end{equation*}
and so $\gamma = 0$. Consequently, we have
		\begin{equation}\label{eq:tok3CN}
		 L_+ u = \beta Q_m.
		\end{equation}
By combining \eqref{eq:tok3CN}, \eqref{eq:chi} with \eqref{eq:L+sp}, \eqref{eq:L+sp1}, we have that 
$u-\beta\chi\in Ker(L_+)=span[\partial_xQ_m]$. Hence, there exists a constant $c \in \R$ such that 
		\begin{equation}\label{eq:u1}
			u = \beta \chi + c \partial_x Q_m.
		\end{equation}
		Since we have that $(\chi, Q_m) < 0$, by taking the scalar product of \eqref{eq:u1} with $Q_m$, we get $\beta = 0$. In particular, $u = c\partial_x Q_m$. This is a contradiction because $(u,\partial_x Q_m) = 0$.
\newline
		Thus we have proven that if $f \in H^1(\T) \cap \mathcal{A} $ is real-valued and $(f, Q_m) = (f,\partial_xQ_m) = 0$ then 
		\begin{equation}\label{eq:L+coerc}
		 (L_+f,f) \geq \tau \| f\|_{L^2}^2.
		\end{equation}
 where $\tau >0$. 	\newline
By combining \eqref{eq:L_min_coerc} and \eqref{eq:L+coerc} we have that, for any $f=f_R+\ii f_I \in \mathcal X$, it holds
\begin{equation}\label{eq:l2pos}
	(L_+ f_R, f_R) + (L_-f_I, f_I) \geq C_0 \| f\|_{L^2}^2,
\end{equation}
where $C_0=\min(C_-,\tau)$. \newline
Finally, we claim that \eqref{eq:l2pos} implies the stronger bound \eqref{eq:H1_pos}.
Let $p = \omega - 3Q_m^2 \in L^\infty(\R)$, $f \in \mathcal{X}$ and 
$c>0$ be a constant that will be determined later.
Using \eqref{eq:l2pos} we can write 
	\begin{equation*}
	 \begin{split}
	 (L_+ f, f) &= c(L_+ f, f) + (1 - c) (L_+ f, f) \\
			&\geq c \| \partial_x f\|_{L^2}^2 - c \| p\|_{L^\infty} \| f\|_{L^2}^2 + (1 - c) C_0 \| f\|_{L^2}^2 = c \| f\|_{H^1}^2 ,
	 \end{split}
	\end{equation*}
	where we have chosen $c = \frac{C_0}{C_0 + \| p\|_{L^\infty} + 1}$. A similar computation can be done for $L_-$.
\end{proof}	

\begin{remark}
We notice that the constant $C$ in \eqref{eq:H1_pos} is uniform in $m\in[0, m_0]$. In particular, it does not go to zero as $m\to0$. 
On the other hand, the same property cannot hold in $H^1\cap\mathcal A$ anymore. Indeed, for any
\begin{equation*}
\xi\in\mathcal X_1=\{u\in H^1(\T)\cap\mathcal A:(u, Q_m)=(u, \ii Q_m)=0\},
\end{equation*}
we have
\begin{equation*}
(L_+\xi_R, \xi_R)+(L_-\xi_I, \xi_I)\geq c \lambda_2(k(m))\|\xi\|_{H^1}^2,
\end{equation*}
where $\lambda_2(k)$ is given in \eqref{eq:L-sp}. In particular, $\lambda_2(k(m))\to0$ as $m\to0$.
\end{remark}
 
The coercivity property \eqref{eq:H1_pos} can now be exploited to provide a similar lower bound for the Lyapunov functional $\mathcal L$ defined in \eqref{eq:L_Lyapunov}.

\begin{lem}\label{thm:L_equivalence} 
	Let $\psi_0 \in \A$, $\psi \in C([0,\infty), \A)$ be its associated solution and let $m$, $\gamma$ and $\xi$ be defined as in \eqref{eq:mass_diss}, \eqref{eq:gamma} and \eqref{eq:decomCN}, respectively. 
Then there exist $\delta>0$ and $C=C(\delta)>0$ not depending on $m$ such that, for $\|\xi(t)\|_{H^1}\leq\delta\sqrt{m(t)}$, we have
	\begin{equation} \label{eq:L_eq} 
		\|\xi(t) \|_{H^1}^2 \leq C \mathcal{L}[\psi(t)].
	\end{equation}
\end{lem}

\begin{proof}
 We decompose $\xi$ as
	\begin{equation}\label{eq:tokk200}
		\xi = a Q_m + \ii b Q_m + y,
	\end{equation}
 where $a,b \in \R$ and $y \in \mathcal{X}$, that is $(y,Q_m) = (y,\ii Q_m) = 0$. 
 From Proposition \ref{thm:pos}, there exists $C>0$ such that
 \begin{equation}\label{eq:tokk203}
 (L_+ y_R, y_R) + (L_- y_I, y_I) \geq C\| y\|_{L^2}^2
 \end{equation}
 where $y_R$ and $y_I$ denote the real and imaginary part of $y$, respectively. Condition \eqref{eq:ortho_con} implies that $b=0$, while by \eqref{eq:ortho_con_2} we have that
 \begin{equation*}
 a = \frac{(\xi, Q_m)}{\| Q_m\|_{L^2}^2} = - \frac{\| \xi\|_{L^2}^2}{2 \| Q_m\|_{L^2}^2},
 \end{equation*}
 and consequently
 \begin{equation}\label{eq:tokk202}
 \| y\|_{L^2}^2 = \| \xi\|_{L^2}^2 - \frac{\| \xi\|_{L^2}^4}{4 \| Q_m\|_{L^2}^2}.
 \end{equation}
	We plug decomposition \eqref{eq:tokk200} into \eqref{eq:L_Lyapunov} to get
\begin{equation}\label{eq:L_prop}
			\begin{split}
				2\mathcal{L} [\psi]& = (L_+ y_R, y_R) + (L_- y_I, y_I) + a^2 (L_+ Q_m, Q_m) + 2a (L_+ Q_m, y) + R(\xi)
			\end{split}
		\end{equation} 
where 
\begin{equation*}
 R(\xi) = - 2\int Q_m \xi_R^3 + Q_m \xi_I^2 \xi_R + \frac{1}{4} |\xi|^4 dx.
\end{equation*}
We use \eqref{eq:tokk202} to bound the rest as
 \begin{equation*}
 \left| R(\xi) \right| \leq K\| \xi\|_{H^1}^3.
 \end{equation*}
	From \eqref{eq:tokk203}, \eqref{eq:tokk202}, Proposition \ref{thm:pos}, after straightforward computations, we obtain that
	\begin{equation*}
 \begin{aligned}
 &2\mathcal{L} [\psi] \geq C \| y \|_{L^2}^2 - 2 a^2 \| Q_m\|_{L^4}^4 + 4 a(Q_m^3,\xi) - K\| \xi\|_{H^1}^3 \\
 &\geq C\left(1 - \frac{1}{4} \frac{ \| \xi \|_{L^2}^2}{ \| Q_m\|_{L^2}^2} \right) \| \xi \|_{L^2}^2 - \frac{\| \xi\|_{L^2}^4}{2 \| Q_m\|_{L^2}^4} \| Q_m\|_{L^4}^4 + 2 \frac{\| \xi\|_{L^2}^2}{ \| Q_m\|_{L^2}^2} (Q_m^3,\xi) - K\| \xi\|_{H^1}^3 
 \end{aligned}
		\end{equation*}
	Hence, there exists $\delta>0$ small enough such that, for $\|\xi\|_{H^1} \leq \delta \| Q_m\|_{L^2}$, we have 
 \begin{equation*}
			\mathcal{L} [\psi] \geq \frac{C}2 \| \xi \|_{L^2}^2 
\end{equation*}
The above inequality can be improved to \eqref{eq:L_eq} 
similarly as done in the concluding step in the proof of Proposition \ref{thm:pos}, see \eqref{eq:l2pos} and subsequent discussion.
\end{proof}

 The next step in our stability analysis concerns the derivation of a suitable bound in time for the functional $\mathcal L$ defined in \eqref{eq:L_Lyapunov}. By using the identity
	 \begin{equation*}
		 \begin{split}
			S'_m[\psi] = \omega \psi - \partial_{xx}\psi - |\psi|^{2} \psi,	
		\end{split}
	\end{equation*}
 we write the damped NLS equation \eqref{eq:Nls_damp} as 
 \begin{equation*}
	 \partial_t \psi = -\ii S'_m[\psi] + \ii\omega \psi - \varepsilon \psi.
 \end{equation*}
 By differentiating $\mathcal L[\psi(t)]$ with respect to time, we obtain
	\begin{equation*}
		\begin{split}
			\frac{d}{dt} \mathcal{L}[\psi] &= \left( S'_m[\psi], \partial_t \psi \right) + \left(S'_m[Q_m], \partial_t Q_{m(t)} \right) - \frac{d\omega(m(t))}{dt} \left( M[\psi] - M[Q_m] \right) \\
			& = - \varepsilon\left( S'_m[\psi], \psi 	\right) = - \varepsilon\left(\omega \| \psi\|_{L^2}^2 + \| \partial_x \psi \|_{L^2}^2 - \| \psi\|_{L^4}^4 \right).
		\end{split}
	\end{equation*}
	Notice that $Q_m$ satisfies equation \eqref{eq:sol_wave} and thus by taking the scalar product of this equation with $Q_m$ we obtain 
	\begin{equation*}
		\omega = \frac{\| Q_m\|_{L^4}^4 - \| \partial_x Q_m\|_{L^2}^2}{\| Q_m\|_{L^2}^2},
	\end{equation*}
 so that, by also recalling that $\psi(t)$ and $Q_{m(t)}$ have the same mass, we infer
 \begin{equation*}
	\begin{split}
		\frac{d}{dt}\mathcal{L}[\psi] & = -\varepsilon( \| \partial_x \psi \|_{L^2}^2 - \| \psi\|_{L^4}^4 + \| Q_m\|_{L^4}^4 - \| \partial_x Q_m\|_{L^2}^2) \\ &
		= -2\varepsilon \mathcal{L}[\psi] + \frac{\varepsilon}{2} \left( \| \psi \|_{L^4}^4 - \| Q_m\|_{L^4}^4 \right).
	\end{split}
 \end{equation*}
 We exploit identity \eqref{eq:decopsi} in the quartic terms above to obtain
	\begin{equation}\label{eq:L_der}
		\begin{split}
			\frac{d}{dt} \mathcal{L}[\psi] &= - 2 \varepsilon \mathcal{L}[\psi] +2 \varepsilon \int_0^{2\pi} Q_m^3 \xi_R dx \\& + \varepsilon \int_0^{2\pi} 3 \xi_R^2 Q_m^2 + \xi_I^2 Q_m^2 + 2\xi_R|\xi|^2 Q_m 	+ \frac{1}{2} |\xi|^4 dx.
		\end{split}
	\end{equation}	
We can now use \eqref{eq:L_der} to obtain a Gronwall-type estimate on $\mathcal L$. 

\begin{prop} \label{thm:L_weak} 
		Let $\psi_0 \in \A$, $\psi \in C([0,\infty), \A)$ be the corresponding solution to \eqref{eq:Nls_damp} and $\xi$ be defined as in \eqref{eq:decomCN}.
Then, for $\|\xi(t)\|_{H^1}\leq\delta\sqrt{m(t)}$, where $\delta>0$ is determined in Lemma \ref{thm:L_equivalence}, we have
		\begin{equation}\label{eq:weak_L_law}
			\mathcal{L}[\psi(t)] \lesssim ( \mathcal{L}[\psi_0] + M[\psi_0]^2 ) e^{-2\varepsilon t}.
		\end{equation} 	
\end{prop}	

\begin{proof} 
Let $m$ be defined as in \eqref{eq:mass_diss} and let $\delta >0$ be the constant determined in Lemma \ref{thm:L_equivalence}.
We estimate the second line in \eqref{eq:L_der} as follows
\begin{equation*}
\begin{aligned}
 \varepsilon \int_0^{2\pi} 3 \xi_R^2 Q_m^2 + \xi_I^2 Q_m^2 + 2\xi_R|\xi|^2 Q_m 	+ \frac{1}{2} |\xi|^4 dx &\lesssim \eps \| \xi\|_{L^{\infty}}^2\big( \| Q_m\|_{L^2}^2 +(\xi,Q_m) + \| \xi\|_{L^2}^2 \big) \\ 
 &\lesssim \eps m\|\xi\|_{H^1}^2\lesssim
 \varepsilon m \mathcal{L}[\psi],
\end{aligned}
\end{equation*}
where the last inequality follows from \eqref{eq:L_eq}. The second term on the right-hand side of \eqref{eq:L_der} may be controlled as 
\begin{equation}\label{eq:obstr}
			2 \varepsilon \int_0^{2\pi} Q_m^3 \xi_R dx \lesssim \varepsilon \|Q_m\|_{L^\infty} \|Q_m\|_{L^2}^2 \| \xi\|_{H^1} \lesssim \varepsilon m \sqrt{\mathcal{L}[\psi]}.
\end{equation}
By plugging those two bounds in \eqref{eq:L_der} we obtain that
		\begin{align*}
			\frac{d}{dt} \mathcal{L}[\psi] \leq ( -2 \varepsilon + \varepsilon C m)\mathcal{L}[\psi] + K \varepsilon m \sqrt{\mathcal{L}[\psi] },
		\end{align*}
for some $C, K>0$. Let us define 
$N(t) = e^{\varepsilon t} \sqrt{\mathcal{L}[\psi(t)]}$, then
		\begin{equation*}
			\frac{d}{dt} N(t) \leq \frac{1}{2} \varepsilon C m(t) N(t) + \frac{1}{2} \varepsilon K m(t) e^{\varepsilon t}.
		\end{equation*}
Gronwall's lemma then yields
		\begin{displaymath}
			N(t) \lesssim m(0) \left(N(0) + (1 - e^{-\varepsilon t} )\right) \leq m(0) N(0) + m(0)
		\end{displaymath}
		which implies \eqref{eq:weak_L_law}.
	\end{proof} 
Let us remark that the estimate \eqref{eq:weak_L_law} provides an improvement on the first bound \eqref{eq:bnd1}. Indeed, for $\psi_0$ sufficiently close in $H^1(\T)$ to $Q_{m(0)}$, we obtain
\begin{equation*}
\| \xi(t)\|_{H^1} \lesssim M[\psi_0] e^{-\varepsilon t}.
\end{equation*}
However, as before this shows the stability of the family of cnoidal waves only if the initial mass is sufficiently small.
Unfortunately, the estimate proved in Proposition \ref{thm:L_weak} cannot be improved, because of the second term on the right-hand side of \eqref{eq:L_der}, see \eqref{eq:obstr}. To 
get rid of this difficulty, in the next section, we modify the profile around which we consider our perturbation. \newline
Let us observe that, to derive the bound \eqref{eq:weak_L_law}, no control on the modulational parameter $\gamma(t)$ is needed. On the other hand, in the next section, the perturbation analysis around the modified profile will require such control. For this purpose, in the next lemma, we provide an estimate on $\left| \omega - \dot \gamma\right|$. In order to derive this bound, we first find the equation satisfied by the remainder $\xi$. 
 By using the definition \eqref{eq:decopsi}, we see that 
\begin{equation} \label{eq:xi}
\begin{aligned}
 \ii \partial_t \xi & = - \partial_{xx}(Q_m + \xi) - |Q_m + \xi|^{2} (Q_m + \xi)
 -\dot\gamma (Q_m + \xi) - \ii \varepsilon \xi -\ii \varepsilon \left(Q_m - 2 m \partial_m Q_m \right).
\end{aligned}
\end{equation}
By exploiting the equation above, we prove the following.

\begin{lem}\label{lem:l-gamma}
There exists $\delta^{(1)} >0$ such that if $\|\xi(t)\|_{L^2} < \delta^{(1)}e^{-\varepsilon t}$ then we have
\begin{equation}\label{eq:omega-gamma}
|\omega(t)-\dot\gamma(t)|\lesssim
\eps+\|\xi(t)\|_{H^1}.
\end{equation}
\end{lem}

\begin{proof}
By taking the scalar product of \eqref{eq:xi} with $Q_m$, we obtain that
\begin{equation}\label{eq:lstima}
\begin{aligned}
0=&\left(Q_m, \ii\d_t\xi\right)
+\left(Q_m, \d_{xx}\xi - \omega \xi
+|Q_m+\xi|^2(Q_m+\xi)-Q_m^3\right)\\
&+(\omega-\dot\gamma)\left(Q_m, Q_m+\xi\right)
+\eps\left(Q_m, \ii(Q_m + \xi) \right) - 2\varepsilon m (Q_m, \ii \partial_m Q_m).
\end{aligned}
\end{equation}
We now analyze the contributions provided by each term in the identity above. By using \eqref{eq:ortho_con}, we rewrite the first term as
\begin{equation*}
\left|\left(Q_m, \ii\d_t\xi\right)\right|=
\left|-\left(\d_tQ_m, \ii\xi\right)\right|=2\eps m \left|\left(\d_mQ_m, \ii\xi\right)\right| \lesssim \varepsilon m,
\end{equation*}
where we used \eqref{eq:bnd1} and $\| \d_mQ_m\|_{L^2} \lesssim \frac{1}{\sqrt{m}}$, see \eqref{eq:dm_Qm}. 
For the second term we have
\begin{equation*}\label{eq:lstima1}
\begin{aligned}
&\big(Q_m,\, \d_{xx}\xi-\omega\xi
+|Q_m+\xi|^2(Q_m+\xi)-Q_m^3\big)\\
&=\big(Q_m, -S''_m[Q_m]\xi
+|\xi|^2(Q_m+\xi)+2Re(\overline{Q}_m\xi)\xi \big)\\
&=\big(-S''_m[Q_m]Q_m, \xi\big)
+\big(Q_m, |\xi|^2(Q_m+\xi)+2Re(\overline{Q}_m\xi)\xi\big)\\
&=\big(2Q_m^3, \xi\big) +\big(Q_m, |\xi|^2(Q_m+\xi)+2Re(\overline{Q}_m\xi)\xi\big).
\end{aligned}\end{equation*}
Since
\begin{equation*}
 \big(2Q_m^3, \xi\big) \leq \| Q_m\|_{L^{\infty}} \| \xi\|_{L^\infty} \| Q_m\|_{L^2}^2 \lesssim C(m_0) m \| \xi\|_{H^1}
\end{equation*}
and
\begin{equation*}
 \big(Q_m, |\xi|^2(Q_m+\xi)+2Re(\overline{Q}_m\xi)\xi\big) \lesssim \| \xi\|_{L^\infty}^2 \| Q_m\|_{L^2}^2 \lesssim m \| \xi\|_{H^1}^2,
\end{equation*}
then we obtain that
\begin{equation*}
 \left|\big(Q_m,\, \d_{xx}\xi-\omega\xi
+|Q_m+\xi|^2(Q_m+\xi)-Q_m^3\big)\right| \lesssim m \left( \| \xi\|_{H^1} + \| \xi\|_{H^1}^2\right).
\end{equation*}
We now consider the third term in \eqref{eq:lstima} and we use \eqref{eq:ortho_con_2} to write
\begin{equation*}
(\omega-\dot \gamma)\left(Q_m, Q_m+\xi\right)
=(\omega-\dot\gamma)\big(2m-\frac{\|\xi\|_{L^2}^2}{2}\big).
\end{equation*}
By \eqref{eq:ortho_con}, we know that the fourth term vanishes. Finally, since $Q_m$ is real-valued, the fifth term vanishes too.
By plugging everything together, we obtain that
\begin{equation}\label{eq:o-l}
\left|\omega(t)-\dot\gamma(t)\right|\big|2m(t)-\frac{\|\xi(t)\|_{L^2}^2}{2}\big|\lesssim m(t)(\eps+\|\xi(t)\|_{H^1}).
\end{equation}
Now we suppose that $\|\xi(t)\|_{L^2}\leq \delta^{(1)} e^{-\varepsilon t}$. From \eqref{eq:mass} we observe that if $\delta^{(1)}$ is small enough, then 
\begin{equation*}
 \big|2m(t)-\frac{\|\xi(t)\|_{L^2}^2}{2}\big| \geq m(t),
\end{equation*}
for any $t >0$. 
Thus \eqref{eq:o-l} yields the estimate \eqref{eq:omega-gamma}. 
\end{proof}

\section{The approximated solitary wave and improved error bounds
} \label{sec:Adj}
The goal of this section is to improve on the control \eqref{eq:weak_L_law} for the perturbation $\xi$. 
The forcing term on the right-hand side of \eqref{eq:xi} implies that the perturbation can grow for times of order $\varepsilon^{-1}$. The presence of this term in \eqref{eq:xi} comes from the fact that the periodic wave is not a solution to \eqref{eq:Nls_damp} for $\eps>0$. 
To overcome this issue we are going to construct an approximating profile $Q_{m, \eps}$ that comprises the forcing at a first order.
Then, we will adapt the analysis developed in Section \ref{sec:first} and prove the coercivity of the linearized operator around $Q_{m, \eps}$. Moreover, we will also show that $Q_{m, \eps}$ is close to $Q_m$, indeed it can be regarded as a first-order correction to the periodic wave. Finally, we will define a modified Lyapunov functional, depending on $Q_{m, \eps}$ that satisfies an improved Gronwall-type estimate. \newline
Let us now explain how we construct the approximate profile $Q_{m, \eps}$. The main idea is to incorporate the forcing term appearing in \eqref{eq:xi} into the equation for $Q_{m, \eps}$, namely we consider the following equation for the profile
\begin{equation}\label{eq:Qmeps}
\partial_{xx}Q_{m, \eps}-\lambda_{m, \eps}Q_{m, \eps}+|Q_{m, \eps}|^2Q_{m, \eps}
+\ii\eps\left(Q_m-2m\partial_mQ_m\right)=0,
\end{equation}
where $\lambda_{m, \eps}$ is a suitable small perturbation of $\omega$. 
We determine the solution to \eqref{eq:Qmeps} employing the Implicit Function theorem.

\begin{prop}\label{thm:impl} 
		There exists $\varepsilon_0 > 0$ such that for any $|\varepsilon| < \varepsilon_0$, there exists $\lambda = \lambda(m,\varepsilon) \in \R$ and a profile $ Q_{m,\varepsilon} \in \A $ satisfying \eqref{eq:Qmeps}. Moreover, we have
		\begin{equation}\label{eq:Q_close}
				\| Q_{m,\varepsilon} - Q_m\|_{H^1} \lesssim \varepsilon \sqrt{m}, \quad |\omega - \lambda| \lesssim \varepsilon\sqrt{m}.
		\end{equation}
\end{prop}

\begin{proof} 
Let us consider the map $G:H^1(\T)\times\R\times\R_+\to H^{-1}\times\R$, defined by
 \begin{equation}\label{eq:F_big}
 G(f,\lambda, \varepsilon) = \left(\begin{matrix}
 -\partial_{xx} f - |f|^2 f + \lambda f - \ii \varepsilon (Q_m - 2m\partial_m Q_m)\\
 M[f] - m
 \end{matrix} \right).
 \end{equation}
The zeroes of this map determine the approximated profile, since
\begin{equation*}
G(Q_{m, \varepsilon}, \lambda_{m, \varepsilon}, \varepsilon)=0
\end{equation*}
implies that \eqref{eq:Qmeps} is satisfied, with $M[Q_{m, \eps}]=m$. However, we remark that for $\varepsilon = 0$, $G$ does not admit a unique solution, since for any $\phi \in [0,2\pi)$, we have
 \begin{equation*}
 G(e^{\ii \phi} Q_m, \omega, 0) = 0.
 \end{equation*}
This implies that the Jacobian matrix of $G$ evaluated in $(Q_m, \omega, 0)$ has a non-trivial kernel and we cannot directly exploit the Implicit Function theorem to find the approximating profile $Q_{m,\varepsilon}$. To overcome this degeneracy at $\varepsilon=0$, we restrict our analysis to the space
\begin{equation}\label{eq:mathY}
		 \mathcal{Y} = \{ u \in \A, \, (u, \ii Q_m) = 0\}.
\end{equation}
Thus, Proposition \ref{thm:cnoidal} yields that for $\varepsilon = 0$, \eqref{eq:F_big} has a unique solution, given by $f = Q_m$, $\lambda = \omega$. \newline
The Jacobian matrix of the function 
\begin{equation*}
 G_0(f,\lambda) = G(f,\lambda, 0)
\end{equation*}
is given by 
\begin{displaymath}
			DG_0(f, \lambda) = \left[\begin{matrix}
				& \mathcal M (f,\lambda) & f & \\
				& (f,\cdot) &0 \\
			\end{matrix} \right],
\end{displaymath}
	where 
\begin{equation*}
 \mathcal M (f,\lambda) u = - \partial_{xx}u + \lambda u - |f|^2 u - 2 Re(\bar{f} u) f.
\end{equation*}
For any $u,z \in \mathcal Y$, $\lambda,\mu \in \R$, we can explicitly write
\begin{align*}
			G_0(z,\lambda) & = G_0(u,\mu) + DG_0(u,\mu)\left( \begin{matrix}
				z-u \\
				\mu- \lambda 
			\end{matrix}\right) \\ & + \left( \begin{matrix}
				(z - u) (\mu - \lambda) + z(|u|^2 - |z|^2) + 2u ( Re(\bar{u}z) - |u|^2) \\
				0 \end{matrix} \right).
\end{align*}
This implies that $G_0$ is continuously differentiable, since
\begin{equation*}
	\begin{split}
				&\bigg \| G_0(z,\mu) -	 G_0(u,\lambda) - DG_0(u,\lambda) \left( \begin{matrix}
					z-u \\
					\mu- \lambda 
				\end{matrix}\right) \bigg\|_{H^1} \\&\lesssim |\mu - \lambda|\| z - u\|_{H^1} + \| z - u\|_{H^1} \big( \| z\|_{H^1} \| z + u \|_{H^1} + \| u\|_{H^1}^2 \big).
	\end{split}
\end{equation*}	
Let us now prove that the Jacobian matrix evaluated in $(Q_m,\omega)$ is invertible. 
To prove the injectivity of $DG_0(Q_m, \omega)$ we suppose that 
		\begin{equation*} 
			DG_0(Q_m, \omega) \left( \begin{matrix}
				\rho \\
				\alpha
			\end{matrix}\right) = 0, 
		\end{equation*}
 for some $\rho \in \mathcal{Y}$ and $\alpha \in \R$. This implies that
\begin{equation}\label{eq:eq1}
		S_m''[Q_m] \rho +\alpha Q_m = 0, \quad (Q_m, \rho) = 0.
 \end{equation}
From the second condition in \eqref{eq:eq1}, we observe that $\rho\in\mathcal X$, see \eqref{eq:X_sp}, consequently Proposition \ref{thm:pos} implies that there exists $C>0$ such that
		\begin{equation*}
		 (S_m''[Q_m] \rho, \rho) \geq C\| \rho\|_{H^1}^2,
		\end{equation*}
see also \eqref{eq:quad}.
		By taking the scalar product of the first equation in \eqref{eq:eq1} with $\rho$, we obtain $\rho = 0$, which implies that $DG_0(Q_m, \omega)$ is injective. To prove its surjectivity, we consider $(w,p) \in \mathcal{Y} \times \R$ and we look for $(\rho, \alpha) \in \mathcal{Y} \times \R$ satisfying 
		\begin{equation}\label{eq:tok100}
		 S_m''[Q_m] \rho + \alpha Q_m = w, \quad (Q_m, \rho) = p. 
		\end{equation}
We take the scalar product of the first equation in \eqref{eq:tok100} with $\partial_m Q_m$ and use
		\begin{equation*} 
			L_+ \partial_m Q_m = - \left(\frac{d}{dm} \omega(m)\right) Q_m,
		\end{equation*} 
to obtain that
		\begin{displaymath}
			-\left(\frac{d}{dm} \omega(m)\right) (\rho,Q_m) + \alpha(Q_m, \partial_m Q_m) = (w, \partial_m Q_m),
		\end{displaymath}
which implies that $\alpha$ can be determined by
	\begin{equation*}
	 \alpha = \frac{\left(\frac{d}{dm} \omega(m)\right) p + (w, \partial_m Q_m) }{(Q_m, \partial_m Q_m)}.
	\end{equation*}
 Moreover, by recalling that $S_m''[Q_m]$ is invertible in $ \mathcal Y$, see \eqref{eq:L-sp} and \eqref{eq:L+sp}, we have
		\begin{displaymath}
			\rho = \left( S_m''[Q_m]\right)^{-1} [ w -\alpha Q_m].
		\end{displaymath}
It is possible to verify that $(S_m''[Q_m])^{-1}:\mathcal{A}^+ \rightarrow \mathcal{A}^+ \cap H^2_{loc}$ (see \cite[Section 2]{GuCoTs17}). Therefore, the Implicit Function theorem implies that there exists a $\varepsilon_0>0$ so that for all $\varepsilon \in [0,\varepsilon_0)$ there exists a profile $Q_{m,\varepsilon} \in \mathcal Y$, and $\lambda_{m, \eps} \in \R$ such that $G(Q_{m,\varepsilon}, \lambda_{m, \varepsilon}, \varepsilon) = 0$. \newline
We now prove the estimates in \eqref{eq:Q_close}.
By construction, the profile $Q_{m,\varepsilon}$ satisfies the equation
		\begin{equation}\label{eq:F_eq}
			- \partial_{xx} Q_{m,\varepsilon} - |Q_{m,\varepsilon}|^2 Q_{m,\varepsilon} + \lambda_{m, \eps} Q_{m,\varepsilon} - \ii \varepsilon(Q_m - 2m \partial_m Q_m) = 0,
		\end{equation}
with $M[Q_{m,\varepsilon}] = m$. 
By denoting
		\begin{equation}\label{eq:zeta}
			\zeta = Q_{m,\varepsilon} - Q_m,
		\end{equation}
we see that it satisfies the following equation 
		\begin{equation}\label{eq:tok101}
 \begin{aligned}
 \ii \varepsilon(Q_m - 2m \partial_m Q_m) &= S''_m[Q_m]\zeta - 2 Re(\zeta) \zeta Q_m -|\zeta|^2 Q_m - |\zeta|^2 \zeta \\
 & + (\lambda_{m, \eps} -\omega) Q_{m,\varepsilon}.
 \end{aligned}
		\end{equation}
The condition $ M[Q_m] = M[Q_{m,\varepsilon}]$ implies that $-2(\zeta,Q_{m,\varepsilon}) = \| \zeta\|_{L^2}^2$ and, since $(Q_{m,\varepsilon}, \ii Q_m) = 0$, we also have that $(\zeta, \ii Q_m) = 0$. Thus, by taking the scalar product of \eqref{eq:tok101} with $\zeta$ we obtain
		\begin{equation}\label{eq:F_eq_2}
			\begin{split}
					(S_m''[Q_m] \zeta, \zeta) &+ (- 2 \zeta_R \zeta Q_m - |\zeta|^2 Q_m - |\zeta|^2 \zeta,\zeta) -\frac{1}2 (\lambda_{m, \eps} - \omega) \| \zeta\|_{L^2}^2 \\ &=\varepsilon (\ii \left(Q_m - 2 m \partial_m Q_m \right), \zeta).
			\end{split}
		\end{equation} 
Lemma \ref{thm:L_equivalence} implies that there exists $C>0$ such that
 \begin{equation*}
 C\| \zeta\|_{H^1}^2 \leq (S_m''[Q_m] \zeta, \zeta).
 \end{equation*}
The cubic terms in $\zeta$ in \eqref{eq:F_eq_2} are estimated as
 \begin{equation*}
 \left|(- 2 \zeta_R \zeta Q_m - |\zeta|^2 Q_m - |\zeta|^2 \zeta,\zeta) \right| \lesssim \| Q_m\|_{L^\infty} \| \zeta\|_{H^1}^3 + \| \zeta\|_{H^1}^4.
 \end{equation*}
 By continuity property with respect to the control parameter implied by the Implicit Function theorem, we can restrict $\varepsilon_0$ to have $\|\zeta\|_{H^1} \ll 1$ small enough so to obtain the inequality 
 \begin{equation*}
		 \frac{C}2\| \zeta\|_{H^1}^2 \leq (S_m''[Q_m] \zeta, \zeta) + (- 2 \zeta_R \zeta Q_m - |\zeta|^2 Q_m - |\zeta|^2 \zeta,\zeta).
		\end{equation*}
		This leads to 
		\begin{equation*}
		 \frac{C}2\| \zeta\|_{H^1}^2 - \frac{1}2 (\lambda_{m, \eps} - \omega) \| \zeta\|_{L^2}^2 \leq 2m \varepsilon (\ii \partial_m Q_m, \zeta).
		\end{equation*}
		 We can further restrict $\varepsilon_0$ to ensure that $|\lambda_{m, \eps} - \omega| < \frac{C}2$ and obtain
		\begin{equation*}
		 \frac{1}4 C\| \zeta\|_{H^1}^2 \leq 2\varepsilon m \left\|\partial_m Q_m \right\|_{L^2} \|\zeta\|_{L^2}.
		\end{equation*}
Now, the first inequality in \eqref{eq:Q_close} comes from the fact that 
$\|\partial_m Q_m \|_{L^2} \lesssim m^{-\frac{1}{2}}$. For the proof of this last estimate, we refer to the Appendix, see \eqref{eq:dm_Qm}.
		Finally, we take the scalar product of \eqref{eq:tok101} with $Q_{m}$ to obtain
 \begin{equation*}
 \begin{aligned}
 (\lambda - \omega) (2m - \| \zeta\|_{L^2}^2/2) &= (i\varepsilon(Q_m - 2m\partial_m Q_m), Q_m) \\
 &- (S''_m[Q_m]\zeta - 2 Re(\zeta) \zeta Q_m -|\zeta|^2 Q_m - |\zeta|^2 \zeta, Q_m) \\
 & \lesssim \| Q_m\|_{L^2}^2 \| \zeta\|_{H^1}^2 + \| Q_m\|_{L^2} \| \zeta\|_{H^1}^3 + \| \zeta\|_{H^1}^4 
 \end{aligned}
 \end{equation*}
 which leads, for $\varepsilon_0$ small enough, to the second inequality in \eqref{eq:Q_close}.
 \end{proof}

\begin{remark}
Let us remark that, while in Proposition \ref{thm:impl} the profile $Q_{m, \eps}$ is constructed as a first-order correction to $Q_m$, it may also be considered as an approximating dissipative soliton. In particular, $Q_{m, \eps}$ can be sought of as an approximation of $\tilde Q_{m(t)}(x)$, where the profile $\tilde Q_{m(t)}$ satisfies the equation 
\begin{equation*}
\partial_{xx}\tilde Q_{m}-\lambda_{m, \eps}\tilde Q_{m}
+|\tilde Q_{m}|^2\tilde Q_{m}
+\ii\eps\left(\tilde Q_{m}-2m\partial_m\tilde Q_{m}\right)=0,
\end{equation*}
so that $\psi(t, x)= e^{\ii\int_0^t\lambda_{m, \eps}ds}Q_{m(t)}(x)$ is a solution to \eqref{eq:Nls_damp}.
Notice that the rigorous construction of a profile for a similar dissipative NLS type dynamics was done in \cite{HaIbMa21}.
\end{remark}
We now show that the approximating profile $Q_{m,\varepsilon}$ can be viewed as the minimizer of a suitable functional, more precisely we define the modified action as follows
\begin{equation}\label{eq:action_new}
\begin{aligned}
 S_{m,\varepsilon}[u] &= E[u] + \lambda_{m,\varepsilon} M[u] - \varepsilon \left( \ii ( Q_m - 2 m \partial_m Q_m),u \right) \\
 & = \frac{1}2 \| \partial_x u\|_{L^2}^2 - \frac{1}4 \| u\|_{L^4}^4 + 
 \frac{\lambda_{m,\varepsilon}}{2}\|u\|_{L^2}^2 - \varepsilon \left( \ii ( Q_m - 2 m \partial_m Q_m),u \right).
\end{aligned}
\end{equation}
It is straightforward to see that \eqref{eq:Qmeps} is the Euler-Lagrange equation associated with the functional \eqref{eq:action_new}, namely $Q_{m, \eps}$ is a critical point for $S_{m, \eps}$, i.e.,
\begin{equation}\label{eq:qmecrt}
 S'_{m, \eps}[Q_{m, \eps}] = 0.
\end{equation}
Moreover, for $\eta \in H^1(\T)$, we have
\begin{equation}\label{eq:2derL}
 S''_{m,\varepsilon}[Q_{m,\varepsilon}] \eta = -\partial_{xx} \eta + \lambda_{m, \eps} \eta - |Q_{m,\varepsilon}|^2 \eta - 2Re(Q_{m,\varepsilon}\bar{\eta})Q_{m,\varepsilon}.
\end{equation}
By a perturbative argument, it is possible to show a coercivity property also for $S''_{m, \eps}[Q_{m, \eps}]$.
\begin{prop} \label{thm:pos1}
There exists $\varepsilon_1 >0$ such that for any $0 < \varepsilon < \varepsilon_1$, there exists a constant $C>0$ such that for any $m \in [0, \infty)$ and
	\begin{equation}\label{eq:X_sp1}
	\eta \in	\mathcal X_{m, \eps}= \left\{ u \in \A; \, (u,Q_{m,\varepsilon}) = (u,\ii Q_{m,\varepsilon}) = 0 \right\},
	\end{equation}
	we have
		\begin{equation} \label{eq:H1_pos1}
			 (S''_{m,\varepsilon}[Q_{m,\varepsilon}] \eta,\eta) \geq C\| \eta \|_{H^1}^2.
		\end{equation} 
\end{prop} 
\begin{proof}
By recalling the definition of the second variation of the action $S_m$, see \eqref{eq:secvar}, we write 
\begin{equation*}
 \begin{aligned}
 & (S''_{m,\varepsilon}[Q_{m,\varepsilon}] \eta,\eta) = \left( S''_m[Q_m] \eta,\eta\right) + \left( (\lambda - \omega) \eta, \eta\right) \\
 & -\left(|Q_{m,\varepsilon}|^2 \eta + 2Re(Q_{m,\varepsilon}\bar{\eta})Q_{m,\varepsilon} - Q_m^2\eta - 2Re(\eta) Q_m^2, \eta \right).
 \end{aligned}
\end{equation*}
By using the estimates in \eqref{eq:Q_close}, we obtain
\begin{equation*}
\begin{aligned}
 &\left| \left((\lambda - \omega) \eta, \eta \right) \right| + 
 \left| \left(|Q_{m,\varepsilon}|^2 \eta + 2Re(Q_{m,\varepsilon}\bar{\eta})Q_{m,\varepsilon} - Q_m^2\eta - 2Re(\eta) Q_m^2, \eta \right) \right| \lesssim \varepsilon \| \eta\|_{L^2}^2,
\end{aligned}
\end{equation*}
which implies that there exists $C>0$ such that
\begin{equation*}
 (S''_{m,\varepsilon}[Q_{m,\varepsilon}] \eta,\eta) \geq (S''_m[Q_m] \eta,\eta) - C\varepsilon \| \eta\|_{L^2}^2.
\end{equation*}
We now claim that \eqref{eq:H1_pos1} follows from Proposition \ref{thm:pos}. Let us first consider the imaginary part $\eta_I$ of $\eta$. By writing
\begin{equation*}
 \eta_I = a Q_m + y
\end{equation*}
where $(y,Q_m) = 0$ and by exploiting Proposition \ref{thm:pos} we have that
\begin{equation*}
 (L_-\eta_I,\eta_I) = (L_-y,y) \geq K\| y\|_{L^2}^2 = K( \| \eta_I\|_{L^2}^2 - a^2 \| Q_m\|_{L^2}^2).
\end{equation*}
It is straightforward to see that 
\begin{equation*}
 a = \frac{(\eta_I,Q_m)}{\| Q_m\|_{L^2}^2}
\end{equation*}
so the previous bound yields
\begin{equation*}
 (L_-\eta_I,\eta_I) \geq K \left( \| \eta_I\|_{L^2}^2 - \frac{ (\eta_I,Q_m)^2}{\|Q_m\|_{L^2}^2} \right).
\end{equation*}
Since $\eta\in\mathcal X_{m, \eps}$, we have that
\begin{equation*}
 \left|(\eta_I,Q_m)\right| = \left|(\eta,\ii(Q_m - Q_{m,\varepsilon}))\right| \leq c\varepsilon\sqrt{m} \| \eta\|_{L^2},
\end{equation*}
where we used \eqref{eq:Q_close}. This bound readily implies
\begin{equation*}
 (L_-\eta_I,\eta_I) \geq K \left(1 - c^2\varepsilon^2 \right) \| \eta_I\|_{L^2}^2 - Kc^2\varepsilon^2 \| \eta_R\|_{L^2}^2
\end{equation*}
where $\eta_R$ denotes the real part of $\eta$.
With a similar procedure, we can also prove that the real part of $\eta$ satisfies
\begin{equation*}
 (L_+\eta_R,\eta_R) \geq K \left(1 - c^2\varepsilon^2 \right) \| \eta_R\|_{L^2}^2 - c^2\varepsilon^2 \| \eta_I\|_{L^2}^2
\end{equation*}
By choosing $0 <\varepsilon_1$ small enough, we obtain that 
\begin{equation*}
 (S''[Q_m] \eta, \eta) =(L_+\eta_R,\eta_R) + (L_-\eta_I,\eta_I) \geq \tilde{K} \| \eta\|_{L^2}^2 
\end{equation*}
for some $\tilde{K} > 0$ not depending on $m$. Then we can conclude as in Proposition \ref{thm:pos}.
\end{proof}

Our next step is to prove the stability of the family of approximating profiles
\begin{equation*}
 \mathcal F_\varepsilon = \left\{ Q_{m,\varepsilon}, \, m\in [0,\infty) \right\}.
\end{equation*}
Combining the stability of the family $\mathcal F_\varepsilon$ and inequalities \eqref{eq:Q_close} we will be able to prove also the stability of the family of cnoidal waves $\mathcal F$. \newline 
To do so, we decompose a solution $\psi$ as the sum of the modified profile and a perturbation $\eta \in H^1(\T)$ 
\begin{equation}\label{eq:decom_new}
 \psi(t,x) = \left( Q_{m(t), \varepsilon} + \eta(t,x) \right)e^{\ii \gamma(t)}
\end{equation}
where $m$ is defined in \eqref{eq:mass_diss} and $\gamma$ in \eqref{eq:gamma}. Observe that we have
\begin{equation}\label{eq:decodeco}
 Q_{m(t), \varepsilon} = Q_{m(t)} + \zeta(t), \quad \eta(t) = \xi(t) - \zeta(t),
\end{equation}
where $\xi$ is defined in \eqref{eq:decomCN} and $\zeta(t) = Q_{m(t), \varepsilon} - Q_{m(t)}$. Notice that we have chosen to keep $\gamma$ as the phase shift, instead of defining a new function $\varphi$ which minimizes the distance 
\begin{equation}\label{eq:defdis1}
d_{\mathcal F_\varepsilon}(\psi(t))=\inf_{\phi\in[0, 2\pi]}\|e^{-\ii \phi}\psi(t)-Q_{m(t),\varepsilon}\|_{L^2}=\|e^{-\ii \varphi(t)}\psi(t)-Q_{m(t),\varepsilon}\|_{L^2}.
\end{equation}
On one hand, by choosing $\varphi$ instead of $\gamma$ we would obtain the orthogonal condition $(\eta,\ii \Q) = 0$ as in Proposition \ref{prop:ort}, instead of an "almost" orthogonal condition in \eqref{eq:ortho_con1} below. On the other hand, we would be required to find a bound on $|\lambda_{m,\eps} - \dot \varphi|$ instead of $|\lambda_{m,\eps} - \dot \gamma|$. To find a bound on $|\lambda_{m,\eps} - \dot \varphi|$, it is required to prove that $\| \partial_m Q_{m,\eps} \|_{L^2} \lesssim \frac{1}{\sqrt{m}}$, see the discussion below Lemma \ref{lem:l-phi}. This is not an easy task because the $H^1$-control \eqref{eq:Q_close} on $Q_m - \Q$ does not guarantee a priori that $\partial_m Q_m$ and $\partial_m \Q$ are close in $L^2(\T)$. Instead, we will use Lemma \eqref{lem:l-gamma} and the second inequality in \eqref{eq:Q_close} to bound the second difference. This second approach requires to know only that $\| \partial_m Q_m \|_{L^2} \lesssim \frac{1}{\sqrt{m}}$ which is proved in the Appendix. \newline
Analogously to Proposition \ref{prop:ort}, we now state some basic properties for $\eta$. 

\begin{prop}\label{prop:ort1}
Let $\psi_0 \in \A$ with $M[\psi_0]=m_0$, $\psi\in C([0,\infty), \A) $ be its associated solution, $m(t)$ be defined as in \eqref{eq:mass_diss}. Then for any $t\geq0$, we have 
\begin{equation}\label{eq:ortho_con1}
		(\eta(t, \cdot),\ii Q_{m(t),\varepsilon}) = (\eta(t, \cdot), \ii \zeta(t, \cdot)) \lesssim \varepsilon \| \eta(t) \|_{L^2} \sqrt{m(t)} , 
\end{equation}
\begin{equation}\label{eq:ortho_con_21}
 (\eta(t, \cdot), Q_{m(t),\varepsilon}) = - M[\eta(t)].
\end{equation}
\end{prop}

\begin{proof}
As in Proposition \ref{prop:ort}, equation \eqref{eq:ortho_con_21} follows from $M[Q_{m,\eps}] = M[\psi]$. In order to prove \eqref{eq:ortho_con1}, we observe that $Q_{m,\eps} \in \mathcal{Y}$ where $\mathcal{Y}$ is defined in \eqref{eq:mathY} and thus 
 \begin{equation*}
 0 = (Q_{m,\eps}, \ii Q_m) = (Q_m + \zeta, \ii Q_m) = (\zeta, \ii Q_m).
 \end{equation*}
 Consequently, by using the orthogonal condition \eqref{eq:ortho_con}, we have that
 \begin{equation*}
 \begin{aligned}
 (\eta, \ii Q_{m,\eps}) = (\xi - \zeta, \ii(Q_m + \zeta)) = (\xi, \ii \zeta) = (\eta, \ii \zeta) \lesssim \varepsilon\| \eta\|_{L^2} \sqrt{m} 
 \end{aligned}
 \end{equation*}
 where we have used \eqref{eq:Q_close} in the last inequality.
\end{proof}
Let us remark that, as for $\xi$, identity \eqref{eq:ortho_con_21} and the condition on the total mass of $Q_{m(t), \eps}$ straightforwardly yields
\begin{equation}\label{eq:mass_eta}
\|\eta(t)\|_{L^2}\lesssim\sqrt{m(t)}. 
\end{equation}
By considering the modified action $S_{m, \eps}$ in \eqref{eq:action_new}, we accordingly define the following Lyapunov functional
\begin{equation}\label{eq:lyapnew}
\begin{aligned}
\mathcal{L}_{m,\varepsilon}[\eta(t)]=& S_{m,\varepsilon}[Q_{m(t),\varepsilon} + \eta(t)] - S_{m,\varepsilon}[Q_{m(t),\varepsilon}]\\
=&S_{m, \eps}[e^{-\ii\gamma(t)}\psi(t)]-S_{m, \eps}[Q_{m, \eps}].
\end{aligned}
\end{equation}
Let us notice the slight difference in the definition of $\mathcal L_{m, \eps}$ with respect to $\mathcal L$, defined in \eqref{eq:Lyap}. This is because, contrarily to \eqref{eq:sol_wave}, the equation for the profile $Q_{m, \eps}$ is no longer invariant by phase shifts, because of the forcing term.
Analogously to Section \ref{sec:first}, we first show that $\mathcal L_{m, \eps}$ controls the $H^1$-norm of the perturbation $\eta$.
\begin{lem}\label{thm:L_new_eq} 
		Let $\psi_0 \in \A$, $\psi\in C([0,\infty), \A) $ be its corresponding solution and let $ \eta \in C([0,\infty), \A)$ be defined as in \eqref{eq:decom_new}. 
		There exists $\varepsilon_2>0, \delta_0 > 0$ such that if $0<\varepsilon < \varepsilon_2$ and $\| \eta(t)\|_{H^1} \leq \delta_0 e^{-\eps t}$, then we have
		\begin{equation} \label{eq:L_new_eq} 
			\|\eta(t) \|_{H^1}^2 \lesssim \mathcal{L}_{m,\varepsilon}[\eta(t)].
		\end{equation}
\end{lem}

\begin{proof}
By using equation \eqref{eq:Qmeps} and identity \eqref{eq:2derL}, it is possible to check that
		\begin{equation}\label{eq:L_new_dec}
\begin{aligned}
\mathcal{L}_{m,\varepsilon}[\eta(t)]=& 
\frac{1}{2} (S_{m,\varepsilon}''[Q_{m(t),\varepsilon}]\eta(t),\eta(t)) 
-\int\frac14|\eta|^4-|\eta|^2Re(\overline{Q}_{m, \eps}\eta)\,dx\\
=& \frac{1}{2} (S_{m,\varepsilon}''[Q_{m(t),\varepsilon}]\eta(t),\eta(t)) + R(\eta(t)),
\end{aligned}
		\end{equation} 
so that
 \begin{equation}\label{eq:L_new_cub}
 \left| R(\eta) \right| =\left|\int\frac14|\eta|^4+|\eta|^2Re(\bar\eta Q_{m, \eps})\,dx\right|
\lesssim \| \eta\|_{H^1}^3 + \| \eta\|_{H^1}^4.
 \end{equation}
Let us now prove that 
 \begin{equation*}
 (S_{m,\varepsilon}''[Q_{m(t),\varepsilon}]\eta(t),\eta(t)) \gtrsim \| \eta(t) \|_{H^1}^2.
\end{equation*}
We decompose $\eta$ as
\begin{equation}\label{eq:etadeko}
\eta = a Q_{m,\varepsilon} + \ii b Q_{m,\varepsilon} + y,
\end{equation}
where $y\in\mathcal X_{m, \eps}$ defined in \eqref{eq:X_sp1}. By using \eqref{eq:ortho_con_21} and \eqref{eq:ortho_con_2}, we deduce that
\begin{equation*}
a = \frac{(\eta, Q_{m,\varepsilon})}{\| Q_{m,\eps}\|_{L^2}^2} = - \frac{1}2 \frac{\| \eta\|_{L^2}^2}{\| Q_{m,\eps}\|_{L^2}^2} \quad b = \frac{(\eta, \ii Q_{m,\varepsilon})}{\| Q_{m,\eps}\|_{L^2}^2} = \frac{(\eta, \ii \zeta)}{\| Q_{m,\eps}\|_{L^2}^2}.
\end{equation*}
Moreover, by Proposition \ref{thm:pos1} we have
\begin{equation}\label{eq:yProp}
(S_{m,\varepsilon}''[Q_{m(t),\varepsilon}]y,y) \gtrsim \| y \|_{H^1}^2.
\end{equation}
We now claim that a similar bound holds for $\eta$, too. This fact may be proved by reproducing the same steps of the proof of Lemma \ref{thm:L_equivalence}. 
By using decomposition \eqref{eq:etadeko}, we can write
\begin{equation}\label{eq:scalPr}
 \begin{aligned}
 (S_{m,\varepsilon}''[Q_{m,\varepsilon}]\eta,\eta) &= (S_{m,\varepsilon}''[Q_{m,\varepsilon}]y,y) - a^2 (S_{m,\varepsilon}''[Q_{m,\varepsilon}]Q_{m,\eps},Q_{m,\eps}) \\
 & - b^2 (S_{m,\varepsilon}''[Q_{m,\varepsilon}] (\ii Q_{m,\eps}),\ii Q_{m,\eps}) - 2 ab (S_{m,\varepsilon}''[Q_{m,\varepsilon}] (\ii Q_{m,\eps}), Q_{m,\eps}) \\
 &+ 2 a (S_{m,\varepsilon}''[Q_{m,\varepsilon}]Q_{m,\eps},\eta) + 2 b(S_{m,\varepsilon}''[Q_{m,\varepsilon}] (\ii Q_{m,\eps}),\eta).
 \end{aligned}
\end{equation}
We now analyze the contributions provided by each term in the identity above. 
Observe that 
\begin{equation}\label{eq:scalPr0}
 \| y\|_{L^2}^2 = \| \eta \|_{L^2}^2 - \frac{1}4 \frac{\| \eta\|_{L^2}^4}{\| Q_{m,\eps}\|_{L^2}^2} - \frac{(\eta, \ii \zeta)^2}{\| Q_{m,\eps}\|_{L^2}^2}.
\end{equation}
By plugging \eqref{eq:scalPr0} into \eqref{eq:yProp}, we bound the first term on the right-hand side of \eqref{eq:scalPr} as
\begin{equation*}
\begin{aligned}
 (S_{m,\varepsilon}''[Q_{m,\varepsilon}]y,y) & \geq C \| y\|_{L^2}^2 \geq C \left( \| \eta \|_{L^2}^2 - \frac{1}4 \frac{\| \eta\|_{L^2}^4}{\| Q_{m,\eps}\|_{L^2}^2} - \frac{(\eta, \ii \zeta)^2}{\| Q_{m,\eps}\|_{L^2}^2} \right) \\
 &\geq C \left( 1- \frac{\| \eta\|_{L^2}^2}{4\| Q_{m,\eps}\|_{L^2}^2} - \frac{\| \zeta\|_{L^2}^2 }{\| Q_{m,\eps}\|_{L^2}^2} \right) \| \eta\|_{L^2}^2.
\end{aligned}
\end{equation*}
To bound the second term in \eqref{eq:scalPr}, we observe that \eqref{eq:F_eq} implies that 
\begin{equation*}
 S_{m,\varepsilon}''[\Q] \Q = -2 |\Q|^2 \Q + \ii \eps (Q_m - 2m\partial_m Q_m),
\end{equation*}
from which it follows that
\begin{equation*}
 \begin{aligned}
 -a^2 (S_{m,\varepsilon}''[Q_{m,\varepsilon}]Q_{m,\eps},Q_{m,\eps}) &= \left( -\frac{\| \eta\|_{L^2}^2}{2\| \Q\|_{L^2}^2}\right)^2 (2 |\Q|^2 \Q - \ii \eps (Q_m - 2m\partial_m Q_m), \Q) \\ 
 & = \frac{\| \eta\|_{L^2}^4}{4 \| \Q\|_{L^2}^4} \left( 2\| \Q\|_{L^4}^4 - \eps (\ii (Q_m - 2m\partial_mQ_m), \Q) \right) \\
 & = \frac{\| \Q\|_{L^4}^4}{2 \| \Q\|_{L^2}^4} \| \eta\|_{L^2}^4 + 2m\eps ( \ii \partial_m Q_m, \Q)\frac{\| \eta\|_{L^2}^2 }{4 \| \Q\|_{L^2}^4} \| \eta\|_{L^2}^2 \\
 & \geq \frac{\| \Q\|_{L^4}^4}{2 \| \Q\|_{L^2}^4} \| \eta\|_{L^2}^4 - \varepsilon c_1 \frac{ \| \eta\|_{L^2}^2 }{ 4 \| \Q\|_{L^2}^2} \| \eta\|_{L^2}^2
 \end{aligned}
\end{equation*}
where we used that $\Q \in \mathcal{Y}$, where $\mathcal{Y}$ is defined in \eqref{eq:mathY} and that $\| \partial_m Q_m\|_{L^2} \| \Q\|_{L^2} \leq c_1$ for some $c_1 >0$, see \eqref{eq:dm_Qm} in the Appendix. 
In the same way, from 
\begin{equation*}
 S''_{m,\eps}[\Q] (\ii \Q) = - \eps (Q_m - 2m\partial_m Q_m),
\end{equation*}
we also have that 
\begin{equation*}
 \begin{aligned}
 - b^2 (S_{m,\varepsilon}''[Q_{m,\varepsilon}] (\ii Q_{m,\eps}),\ii Q_{m,\eps}) & = \eps\frac{(\eta, \ii \zeta)^2}{\| Q_{m,\eps}\|_{L^2}^4} ( (Q_m - 2m\partial_m Q_m), \ii \Q) \\
 & = \eps\frac{(\eta, \ii \zeta)^2}{\| Q_{m,\eps}\|_{L^2}^2} ( \partial_m Q_m, \ii \Q) \\
 & \geq -\varepsilon c_1 \frac{\| \zeta \|_{L^2}^2}{\| \Q\|_{L^2}^2} \| \eta\|_{L^2}^2.
 \end{aligned}
\end{equation*}
For the fourth term in \eqref{eq:scalPr}, we notice that by multiplying equation \eqref{eq:Qmeps} by $\Q$, integrating and taking imaginary part, we have 
\begin{equation*}
 (\varepsilon(Q_m - 2m\partial_m Q_m), \Q) = 0,
\end{equation*}
which readily implies that
\begin{equation*}
 2 ab (S_{m,\varepsilon}''[Q_{m,\varepsilon}] (\ii Q_{m,\eps}), Q_{m,\eps}) = 2ab (-\varepsilon(Q_m - 2m\partial_m Q_m), \Q) = 0
\end{equation*}
For the last two terms in \eqref{eq:scalPr}, we observe that 
\begin{equation*}
 \begin{aligned}
 2a (S_{m,\varepsilon}''[Q_{m,\varepsilon}]Q_{m,\eps},\eta) &= \frac{\| \eta\|_{L^2}^2}{\| Q_{m,\eps}\|_{L^2}^2} ( 2 |\Q|^2 \Q - \ii \eps (Q_m - 2m\partial_m Q_m),\eta) \\
 & \geq - 2 \frac{\| \eta\|_{L^2}^3 \| \Q\|_{L^\infty}^2}{\| Q_{m,\eps}\|_{L^2}} - \varepsilon c_2 \| \eta \|_{L^2}^2
 \end{aligned} 
\end{equation*}
and 
\begin{equation*}
\begin{aligned}
 2b (S_{m,\varepsilon}''[Q_{m,\varepsilon}] (\ii Q_{m,\eps}),\eta) &= - \eps \frac{2(\eta, \ii \zeta)}{\| Q_{m,\eps}\|_{L^2}^2} ((Q_m - 2m \partial_m Q_m),\eta) \\ 
 &\geq -2 c_2\varepsilon \| \eta \|_{L^2}^2
\end{aligned}
\end{equation*}
where we used \eqref{eq:Q_close} and that there exists a constant $c_2>0$ such that
\begin{equation*}
 \| Q_m - 2m\partial_m Q_m\|_{L^2} \leq c_2\sqrt{m},
\end{equation*}
see \eqref{eq:dm_Qm} in the Appendix. 
Putting everything together we obtain that there exists $C_1 >0$ such that
\begin{equation}\label{eq:scalPr2}
 \begin{aligned}
 (S''_{m,\eps} [\Q] \eta, \eta) &\geq C \left(1- C_1 \eps - \frac{ (1 + C_1 \varepsilon) \| \eta\|_{L^2}^2}{4\| Q_{m,\eps}\|_{L^2}^2} - \frac{(1 + \varepsilon C_1) \| \zeta\|_{L^2}^2 }{\| Q_{m,\eps}\|_{L^2}^2} \right) \| \eta\|_{L^2}^2 \\ 
 & - 2 \frac{\| \eta\|_{L^2}^3 \| \Q\|_{L^\infty}^2}{\| Q_{m,\eps}\|_{L^2}} + \frac{\| \Q\|_{L^4}^4}{2 \| \Q\|_{L^2}^4} \| \eta\|_{L^2}^4.
 \end{aligned}
\end{equation}
To conclude, we observe from \eqref{eq:Q_close} that we have
\begin{equation*}
 \frac{ \| \zeta\|_{L^2}^2 }{\| Q_{m,\eps}\|_{L^2}^2} \lesssim \varepsilon,
\end{equation*}
and also that the last two terms can be bounded as 
\begin{equation*}
 \left|- 2 \frac{\| \eta\|_{L^2}^3 \| \Q\|_{L^\infty}^2}{\| Q_{m,\eps}\|_{L^2}} + \frac{\| \Q\|_{L^4}^4}{2 \| \Q\|_{L^2}^4} \| \eta\|_{L^2}^4 \right|\leq C_3 \left(\| \eta\|_{L^2}^3 + \| \eta\|_{L^2}^4 \right).
\end{equation*}
where $C_3$ is not depending on $m$.
Again, as in Lemma \ref{thm:L_equivalence}, by using \eqref{eq:L_new_dec}, \eqref{eq:L_new_cub} and \eqref{eq:scalPr2}, we may prove that, by choosing $\delta_0>0$ and $\varepsilon_2 >0$ sufficiently small, we obtain that
\begin{equation*}
\mathcal{L}_{m,\varepsilon}[\eta(t)] \geq K(\| \eta\|_{H^1}^2 + R(\eta)) \geq \frac{K}2 \| \eta\|_{H^1}^2,
\end{equation*}
for $\|\eta\|_{H^1}\leq\delta_0 e^{-\eps t}$.
\end{proof}

We now show that the modified Lyapunov functional $\mathcal L_{m, \eps}$ satisfies a better bound with respect to estimate \eqref{eq:weak_L_law} derived for $\mathcal L$, defined in \eqref{eq:Lyap}.
We recall that the introduction of the approximating profile $Q_{m, \eps}$ and the definition of $\mathcal L_{m, \eps}$ stemmed from the necessity to improve on estimate \eqref{eq:weak_L_law}. To justify why we expect a better control on $\mathcal L_{m, \eps}$, we may compare \eqref{eq:xi} with the equation for the perturbation $\eta$. By using \eqref{eq:decom_new}, \eqref{eq:Nls_damp} and \eqref{eq:Qmeps}, we see that $\eta$ satisfies
\begin{equation}\label{eq:eta}
\begin{aligned}
\ii \partial_t\eta&+\d_{xx}\eta-\lambda_{m, \eps}\eta
+|Q_{m, \eps}+\eta|^2(Q_{m, \eps}+\eta)-|Q_{m, \eps}|^2Q_{m, \eps}+\ii\eps\eta\\
&+(\lambda_{m, \eps}-\dot\gamma)(Q_{m, \eps}+\eta)
 +\ii\eps(Q_{m, \eps}-Q_m-2m\partial_m(Q_{m, \eps}-Q_m))=0.
\end{aligned}
\end{equation}
We now see that the forcing term, namely the last term in the equation above, is of order $\eps^2$, as we have
\begin{equation*}
		\begin{split}
			&\|	\varepsilon(Q_m - 2m\partial_m Q_m - Q_{m,\varepsilon} + 2m \partial_m Q_{m,\varepsilon}) \|_{H^1} \\ & \lesssim \varepsilon \left( m\| \partial_m Q_{m,\varepsilon} - \partial_m Q_m\|_{H^1} + \| Q_{m,\varepsilon} - Q_m\|_{H^1} \right) \lesssim \varepsilon^2.
		\end{split}
\end{equation*}

In what follows we provide a Gronwall-type estimate for $\mathcal L_{m, \eps}[\eta(t)]$. As it will be clear from the proof of Proposition \ref{thm:L_bound} below, the improvement with respect to \eqref{eq:weak_L_law} lies in the fact that the linear terms appearing in the time derivative of $\mathcal L_{m, \eps}$ can be estimated by $ \varepsilon^2 m \sqrt{\mathcal{L}_{m,\varepsilon}[\eta(t)]}$. This should be compared with the second term on the right-hand side of \eqref{eq:L_der}, which was bounded by $ \varepsilon m \sqrt{\mathcal{L}_{m,\varepsilon}[\eta(t)]} $.
On the other hand, there is an additional difficulty in estimating the time derivative of $\mathcal L_{m, \eps}$, which is given by providing a suitable control on the phase shift in \eqref{eq:decom_new}. As previously said, this is because the equation for the approximated profile $Q_{m, \eps}$ is no longer invariant by phase shifts. We overcome this difficulty by providing an estimate for $|\lambda_{m, \eps}-\dot \gamma|$, see \eqref{eq:l-phi} below.

\begin{lem}\label{lem:l-phi}
There exists $\delta_1 >0$ and $\varepsilon_3 >0$ such that if $\|\eta(t)\|_{L^2} < \delta_1 e^{-\varepsilon t}$ and $\varepsilon < \varepsilon_3$ then we have
\begin{equation}\label{eq:l-phi}
|\lambda_{m, \eps}(t)-\dot\gamma(t)|\lesssim
\eps+\|\eta(t)\|_{H^1}.
\end{equation}
\end{lem}

\begin{proof}
This bound follows from Lemma \ref{lem:l-gamma} and the second inequality in \eqref{eq:Q_close}. Indeed, Lemma \ref{lem:l-gamma} implies that there exists $\delta^{(1)} >0$ such that for $\| \xi(t)\|_{L^2} <\delta^{(1)} e^{- \varepsilon t}$, we have that 
\begin{equation}\label{eq:omega-gam1}
 |\omega(t) - \dot \gamma(t)| \lesssim (\varepsilon + \| \xi(t)\|_{H^1}).
\end{equation}
Now we suppose that $\|\eta(t)\|_{L^2}\leq \delta_1 e^{-\varepsilon t}$. From \eqref{eq:decodeco} and \eqref{eq:Q_close}, this implies that there exists $C>0$ such
\begin{equation*}
 \|\xi(t)\|_{L^2} \leq C\left(\|\eta(t)\|_{L^2} + \| \zeta(t)\|_{L^2}\right) \leq C(\delta_1 + \varepsilon) e^{-\varepsilon t}
\end{equation*}
We choose $\delta_1$ and $\eps$ is small enough so that $C(\delta_1 + \varepsilon) < \delta^{(1)} $, and thus we obtain inequality \eqref{eq:omega-gam1} for any $t >0$. 
In order to obtain \eqref{eq:l-phi}, we use \eqref{eq:decodeco} and \eqref{eq:Q_close} to have
\begin{equation*}
 \|\xi(t)\|_{H^1} \lesssim \|\eta(t)\|_{H^1} + \| \zeta (t)\|_{H^1} \lesssim \|\eta(t)\|_{H^1} + \varepsilon \sqrt{m(t)},
\end{equation*}
where $\zeta$ is defined in \eqref{eq:zeta}. Moreover, we observe that
\begin{equation*}
 \left|\lambda_{m,\eps}-\dot\gamma\right| \leq \left|\lambda_{m,\eps}-\omega\right| + \left|\omega-\dot\gamma\right| \lesssim \varepsilon(1 + \sqrt{m}) + \| \eta\|_{H^1} \lesssim \varepsilon + \| \eta\|_{H^1} 
\end{equation*}
where we used \eqref{eq:omega-gamma} and \eqref{eq:Q_close}.
\end{proof}

As anticipated, we would like to remark why we kept $\gamma$ as the phase shift in the decomposition \eqref{eq:decom_new} instead of using the more natural $\varphi$ defined in \eqref{eq:defdis1}. On one hand, we would obtain a more direct proof of Lemma \ref{thm:L_new_eq}. On the other hand, we would struggle to find a suitable bound on the difference $|\lambda_{m,\eps} - \dot \varphi|$. Indeed, if we take the scalar product of equation \eqref{eq:eta} with $\Q$, we find the term $(\ii \partial_m \Q, Q_m)$ which is not easy to bound because we do not have the control $ \| \partial_m \Q\|_{L^2} \lesssim \frac{1}{\sqrt{m}}$. On the other hand, we could not use Lemma \ref{lem:l-gamma} because it is not clear how to prove the smallness of $|\dot \gamma - \dot \varphi|$. 
\newline
We will now prove that the new remainder $\eta$ can be bounded with a control that is better than that in \eqref{eq:weak_L_law} and which is enough to show that the remainder remains small for all positive times. By exploiting Lemma \ref{lem:l-phi} and the coercivity property in Lemma \ref{thm:L_new_eq}, we prove the following.

\begin{prop} \label{thm:L_bound}
		Let $\psi_0 \in \A$, $\psi\in C([0,\infty), \A) $ be the corresponding solution to \eqref{eq:Nls_damp}. Let $m$, $\gamma$ and $ \eta \in C([0,\infty), \A)$ be defined as in \eqref{eq:mass_diss}, \eqref{eq:gamma} and \eqref{eq:decom_new}, respectively. 
		There exists $\varepsilon^* = \varepsilon^* (m(0))>0, \delta^* > 0$ such that for any $0<\varepsilon < \varepsilon^*$ and 
 $\| \eta(0)\|_{H^1} \leq \delta^*$, then for any 
 $0 \leq t \leq \eps^{-\frac{3}{2}}$
		\begin{equation}\label{eq:L_bound}
			\mathcal{L}_{m,\varepsilon}[\eta(t)] \leq C\left(\sqrt{\mathcal{L}_{m,\varepsilon}[\eta(0)]}+ \varepsilon^2 t\right)^2 e^{-2\varepsilon t},
		\end{equation}
for some constant $C=C(m_0)>0$.
\end{prop}
 
\begin{proof}
By differentiating with respect to time the functional $\mathcal L_{m, \eps}$, we obtain
\begin{equation}\label{eq:lyap_tder}
\begin{aligned}
\frac{d}{dt}\mathcal L_{m, \eps}[\eta]&=
\left(S'_{m, \eps}[Q_{m, \eps}+\eta], \d_t(Q_{m, \eps}+\eta)\right)
-\left(S'_{m, \eps}[Q_{m, \eps}], \d_tQ_{m, \eps}\right)\\
&+\frac{d}{dt}\lambda_{m, \eps}\left(M[Q_{m, \eps}+\eta]-M[Q_{m, \eps}]\right)
-Re\int\ii\eps\d_t(Q_m-2m\d_mQ_m)\bar\eta\,dx\\
&=\left(S'_{m, \eps}[Q_{m, \eps}+\eta], \d_t(Q_{m, \eps}+\eta)\right)
-Re\int\ii\eps\d_t(Q_m-2m\d_mQ_m)\bar\eta\,dx,
\end{aligned}
\end{equation}
where in the last identity we used \eqref{eq:qmecrt} and the fact that $M[Q_{m, \eps}+\eta]=M[\psi]=M[Q_{m, \eps}]$. Let us notice that the first variation of $S_{m, \eps}$ appearing above may be written also in the following form
\begin{equation}\label{eq:111}
\begin{aligned}
S'_{m, \eps}[Q_{m, \eps}+\eta]=&-\d_{xx}\eta+\lambda_{m, \eps}\eta-|Q_{m, \eps}+\eta|^2(Q_{m, \eps}+\eta)+|Q_{m, \eps}|^2Q_{m, \eps}\\
=&S''_{m, \eps}[Q_{m, \eps}]\eta-2Re(\overline{Q}_{m, \eps}\eta)\eta-|\eta|^2(Q_{m, \eps}+\eta).
\end{aligned}
\end{equation}
Consequently, we may also write equation \eqref{eq:eta} as follows,
\begin{equation*}
 \begin{aligned}
 0 &= \ii\d_t\eta-S'_{m, \eps}[Q_{m, \eps}+\eta]+(\lambda_{m, \eps}-\dot\gamma)(Q_{m, \eps}+\eta)+\ii\eps\eta \\
&+\ii\eps(Q_{m, \eps}-Q_m-2m\partial_m(Q_{m, \eps}-Q_m)).
 \end{aligned}
\end{equation*}
By using the equation above and the fact that $\d_tQ_{m, \eps}=-2\eps m\d_mQ_{m, \eps}$, we obtain
\begin{equation}\label{eq:gr_fst}
\begin{aligned}
\left(S'_{m, \eps}[Q_{m, \eps}+\eta], \d_t(Q_{m, \eps}+\eta)\right) 
& =-2\eps m\left(S'_{m, \eps}[Q_{m, \eps}+\eta], \d_mQ_{m, \eps}\right)\\
& -\eps\left(S'_{m, \eps}[Q_{m, \eps}+\eta], \eta+Q_{m, \eps}-Q_m-2m\d_m(Q_{m, \eps}-Q_m)\right)\\
& +(\lambda_{m, \eps}-\dot\gamma)\left(S'_{m, \eps}[Q_{m, \eps}+\eta], \ii(Q_{m, \eps}+\eta)\right)\\
&=-\eps\left(S'_{m, \eps}[Q_{m, \eps}+\eta], \eta\right)
\\
&-\eps\left(S'_{m, \eps}[Q_{m, \eps}+\eta],Q_{m, \eps}-Q_m\right)\\
& -2\eps m\left(S'_{m, \eps}[Q_{m, \eps}+\eta], \d_mQ_{m}\right) \\
&+(\lambda_{m, \eps}-\dot\gamma)\left(S'_{m, \eps}[Q_{m, \eps}+\eta], \ii(Q_{m, \eps}+\eta)\right).
\end{aligned}\end{equation}
We now estimate \eqref{eq:gr_fst} term by term. For the first one, we use \eqref{eq:111} to write
\begin{equation*}
-\eps\left(S'_{m, \eps}[Q_{m, \eps}+\eta], \eta\right)
=-\eps\left(S''_{m, \eps}[Q_{m, \eps}]\eta, \eta\right)
+\varepsilon\int|\eta|^2Re\left((3Q_{m, \eps}+\eta)\bar\eta\right)\,dx.
\end{equation*}
We thus recall identity \eqref{eq:L_new_dec} to further develop the previous expression as
\begin{equation*}
-\eps\left(S'_{m, \eps}[Q_{m, \eps}+\eta], \eta\right)
=-2\eps \mathcal L_{m, \eps}[\eta]-\varepsilon\int\frac12|\eta|^4+5|\eta|^2Re(\overline{Q}_{m, \eps}\eta)\,dx.
\end{equation*}
The nonlinear term on the right-hand side may be estimated by
\begin{equation*}
\varepsilon \left|\int\frac12|\eta|^4+5|\eta|^2Re(\overline{Q}_{m, \eps}\eta)\,dx\right|\lesssim \varepsilon m\|\eta\|_{H^1}^2.
\end{equation*}
We now consider the second term in \eqref{eq:gr_fst}. By using identity \eqref{eq:111}, we obtain that
\begin{equation*}
 \begin{aligned}
 - \eps \left(S'_{m, \eps}[Q_{m, \eps}+\eta],Q_{m, \eps}-Q_m\right) &= - \eps\left(S''_{m, \eps}[Q_{m, \eps}]\eta, Q_{m, \eps}-Q_m \right) \\ 
 & + \eps \left( 2Re(\overline{Q}_{m, \eps}\eta)\eta+|\eta|^2(Q_{m, \eps}+\eta), Q_{m, \eps}-Q_m \right).
 \end{aligned}
\end{equation*}
The higher order terms are estimated from \eqref{eq:Q_close} and \eqref{eq:mass_eta} as
\begin{equation*}
 \begin{aligned}
 \eps \left( 2Re(\overline{Q}_{m, \eps}\eta)\eta+|\eta|^2(Q_{m, \eps}+\eta), Q_{m, \eps}-Q_m \right) &\lesssim \eps \| \Q - Q_m\|_{L^2} \| \eta\|_{L^\infty}^2 (\|\Q\|_{L^2} + \| \eta\|_{L^2}) \\ &\lesssim \varepsilon^2 m \| \eta \|_{H^1}^2.
 \end{aligned}
\end{equation*}
Moreover, by using identities \eqref{eq:2derL}, \eqref{eq:Qm} and equation \eqref{eq:F_eq}, we observe that 
\begin{equation*}
 \begin{aligned}
 S''_{m, \eps}[Q_{m, \eps}](\Q - Q_m) &= ( \omega - \lambda_{m,\eps}) Q_m + (|\Q|^2 - Q_m^2) Q_m \\ 
 &+ 2\Q( Q_m Re(\Q) - |\Q|^2) \\
 &+ \ii \eps (Q_m - 2m \partial_m Q_m). 
 \end{aligned}
\end{equation*}
Consequently \eqref{eq:Q_close} and \eqref{eq:mass_eta} imply that
\begin{equation*}
 \begin{aligned}
 \eps \left|\left(S''_{m, \eps}[Q_{m, \eps}]\eta, Q_{m, \eps}-Q_m \right) \right| & \lesssim \eps |\omega - \lambda_{m,\eps}| \left|(\eta,Q_m)\right| \\
 &+ \eps \left|\left(\eta, (|\Q|^2 - Q_m^2) Q_m + 2\Q( Q_m Re(\Q) - |\Q|^2) \right) \right| \\
 &+ \eps \left|\left(\eta, \ii \eps(Q_m -2m\partial_m Q_m) \right) \right| \\
 &\lesssim \eps \| \eta\|_{L^2} \sqrt{m} \left(|\omega - \lambda_{m,\eps}| + \sqrt{m}\| \Q - Q_m\|_{H^1} \right) \\
 &+ \eps^2\left|\left(\eta, \ii (Q_m -2m\partial_m Q_m) \right)\right|\\
 &\lesssim \eps^2 m \| \eta\|_{L^2} + \eps^2\left|\left(\eta, \ii (Q_m -2m\partial_m Q_m) \right)\right|.
 \end{aligned}
\end{equation*}
Thus, from \eqref{eq:dm_Qm}, we have that 
\begin{equation*}\begin{aligned}
\left|-\eps\left(S'_{m, \eps}[Q_{m, \eps}+\eta],Q_{m, \eps}-Q_m\right)\right|
\lesssim \varepsilon^2 m \| \eta \|_{H^1}^2 + \eps^2 \sqrt{m}\| \eta\|_{L^2}. 
\end{aligned}\end{equation*}
We use identity \eqref{eq:111} to estimate the third term on the right-hand side of \eqref{eq:gr_fst} and we get
\begin{equation*}
\begin{aligned}
 &-2\eps m\left(S'_{m, \eps}[Q_{m, \eps}+\eta], \d_mQ_{m}\right)\\
 &=-2\eps m\left(S''_{m, \eps}[Q_{m, \eps}]\eta-2Re(\overline{Q}_{m, \eps}\eta)\eta-|\eta|^2(Q_{m, \eps}+\eta), \d_mQ_m\right)\\
 &=-2\eps m\left(\eta, S''_{m, \eps}[Q_{m, \eps}]\d_mQ_m\right)
 +2\eps m\left(2Re(\overline{Q}_{m, \eps}\eta)\eta+|\eta|^2(Q_{m, \eps}+\eta), \d_mQ_m\right).
\end{aligned}
\end{equation*}
Let us focus on the first term. From \eqref{eq:2derL} and \eqref{eq:secvar} we may write
\begin{equation}\label{eq:112}
\begin{aligned}
S''_{m, \eps}[Q_{m, \eps}]\d_mQ_m
&=S''_m[Q_m]\d_mQ_m
+(\lambda_{m, \eps}-\omega(m))\d_mQ_m
-\left(|Q_{m, \eps}|^2-Q_m^2\right)\d_mQ_m\\
&-2\left[Re\left(\overline{(Q_{m, \eps}-Q_m)}\d_mQ_m\right)Q_{m, \eps}
+Re(\bar Q_m\d_mQ_m)(Q_{m, \eps}-Q_m)\right].
\end{aligned}
\end{equation}
Moreover, by differentiating equation \eqref{eq:sol_wave} with respect to $m$ we obtain
\begin{equation}\label{eq:om_pr_Qm}
S''_m[Q_m]\d_mQ_m=L_+\d_mQ_m=-\frac{d\omega}{dm}Q_m.
\end{equation}
By combining \eqref{eq:112} and \eqref{eq:om_pr_Qm}, we obtain that
\begin{equation*}\begin{aligned}
-2\eps m\left(\eta, S''_{m, \eps}[Q_{m, \eps}]\d_mQ_m\right)
&=2\eps m\frac{d\omega}{dm}\left(\eta, Q_m\right)
-2\eps m(\lambda_{m, \eps}-\omega)(\eta, \d_mQ_m)\\
&+2\eps m\left(\eta,(|Q_{m, \eps}|^2-Q_m^2)\d_mQ_m\right)\\
&+4\eps m\Bigg(\eta,Re\left(\overline{(Q_{m, \eps}-Q_m)}\d_mQ_m\right)Q_{m, \eps} \\
&+Re(\bar Q_m\d_mQ_m)(Q_{m, \eps}-Q_m)
\Bigg).
\end{aligned}\end{equation*}
By using that $\frac{d\omega}{dm} \lesssim 1$, see \eqref{eq:om_pr} in the Appendix, we may now estimate the third term in \eqref{eq:gr_fst} to obtain
\begin{equation*}\begin{aligned}
\left|2\eps m\left(S'_{m, \eps}[Q_{m, \eps}+\eta], \d_mQ_m\right)\right|
&\lesssim\eps m\|\eta\|_{L^2}^2+\eps m|\lambda_{m, \eps}-\omega|\|\eta\|_{L^2}\|\d_mQ_m\|_{L^2}\\
&+\eps m\|\d_mQ_m\|_{L^2}\|\eta\|_{H^1}\sqrt{m}\|Q_{m, \eps}-Q_m\|_{H^1}.
\end{aligned}\end{equation*}
Now we use $\|\d_mQ_m\|_{L^2} \lesssim \frac{1}{\sqrt{m}}$, see \eqref{eq:dm_Qm} in the Appendix, and \eqref{eq:Q_close} to obtain 
\begin{equation*}\begin{aligned}
\left|2\eps m\left(S'_{m, \eps}[Q_{m, \eps}+\eta], \d_mQ_m\right)\right|
&\lesssim\eps m\|\eta\|_{L^2}^2+\eps^2 m\|\eta\|_{H^1}.
\end{aligned}
\end{equation*}
Let us now consider the last term in \eqref{eq:gr_fst}. First of all, we write
\begin{equation*}\begin{aligned}
&\left(S'_{m, \eps}[Q_{m, \eps}+\eta], \ii(Q_{m, \eps}+\eta)
\right)\\
&=\left(S''_{m, \eps}[Q_{m, \eps}]\eta-2Re(\overline{Q}_{m, \eps}\eta)\eta-|\eta|^2(Q_{m, \eps}+\eta),
\ii(Q_{m, \eps}+\eta)\right)\\
&=\left(-|\eta|^2(Q_{m, \eps}+\eta), \ii\eta\right)
+\left(\eta, S''_{m, \eps}[Q_{m, \eps}](\ii Q_{m, \eps})\right)\\
&-\left(2Re(\overline{Q}_{m, \eps}\eta)\eta+|\eta|^2(Q_{m, \eps}+\eta),\ii Q_{m, \eps}\right).
\end{aligned}\end{equation*}
By using the definition \eqref{eq:2derL} and equation \eqref{eq:Qmeps}, we may write the above expression as
\begin{equation*}\begin{aligned}
\left(S'_{m, \eps}[Q_{m, \eps}+\eta], \ii(Q_{m, \eps}+\eta)
\right) &= \left(-|\eta|^2(Q_{m, \eps}+\eta), \ii\eta\right)
-\eps\left(\eta,(Q_m-2m\d_mQ_m)\right)\\
&-\left(2Re(\overline{Q}_{m, \eps}\eta)\eta+|\eta|^2(Q_{m, \eps}+\eta),\ii Q_{m, \eps}\right).
\end{aligned}\end{equation*}
We now use Sobolev embedding to obtain that
\begin{equation*}
\left|\left(S'_{m, \eps}[Q_{m, \eps}+\eta], \ii(Q_{m, \eps}+\eta)\right)\right|
\lesssim\eps\sqrt{m} \|\eta\|_{H^1}+m\|\eta\|_{H^1}^2.
\end{equation*}
For the last term in \eqref{eq:gr_fst}, we observe that
\begin{equation*}
 \begin{aligned}
 \left(S'_{m, \eps}[Q_{m, \eps}+\eta], \ii(Q_{m, \eps}+\eta)\right) & = \left(S''_{m, \eps}[Q_{m, \eps}]\eta, \ii Q_{m, \eps} \right) + \left(S''_{m, \eps}[Q_{m, \eps}]\eta, \ii \eta \right) \\ 
 & + \left( 2Re(\overline{Q}_{m, \eps}\eta)\eta+|\eta|^2(Q_{m, \eps}+\eta), \ii(Q_{m, \eps}+\eta) \right).
 \end{aligned}
\end{equation*}
By Sobolev embedding, \eqref{eq:mass_eta} and \eqref{eq:111}, the higher order terms can be estimated as 
\begin{equation*}
 \begin{aligned}
 \left( 2Re(\overline{Q}_{m, \eps}\eta)\eta+|\eta|^2(Q_{m, \eps}+\eta), \ii(Q_{m, \eps}+\eta) \right) \lesssim m \| \eta\|_{H^1}^2.
 \end{aligned}
\end{equation*}
Moreover, from identity \eqref{eq:2derL} we have 
\begin{equation*}
 \begin{aligned}
 \left(S''_{m, \eps}[Q_{m, \eps}]\eta, \ii Q_{m, \eps} \right) &= - (\eta, \eps(Q_m - 2m \partial_m Q_m)) \\
 &\lesssim \eps \sqrt{m} \| \eta \|_{L^2}.
 \end{aligned}
\end{equation*}
Finally, by exploiting \eqref{eq:l-phi}, we have
\begin{equation*}
\begin{aligned}
 |\lambda_{m, \eps}-\dot\gamma|\big|\left(S'_{m, \eps}[Q_{m, \eps}+\eta], \ii(Q_{m, \eps}+\eta)\right)\big|
&\lesssim(\eps \sqrt{m} \|\eta\|_{L^2}+m\|\eta\|_{H^1}^2)(\eps+\|\eta\|_{H^1}) \\
& \lesssim \eps^2 \sqrt{m} \|\eta\|_{L^2} + \eps \sqrt{m} \|\eta\|_{H^1}^2 + m\|\eta\|_{H^1}^3,
\end{aligned}
\end{equation*}
for $\| \eta(t)\|_{L^2} < \delta_1 e^{-\varepsilon t}$ where $\delta_1>0$ is defined in Lemma \ref{lem:l-phi}. 
By collecting all the estimates above, we obtain that there exists $C>0$ such that the first term in \eqref{eq:lyap_tder} can be bounded as
\begin{equation}\label{eq:stima1}
 \begin{aligned}
 \left( S_{m,\varepsilon}'[Q_{m,\varepsilon} + \eta], \partial_t (Q_{m,\varepsilon} + \eta) \right) &\leq-2 \varepsilon \mathcal L_{m,\eps}[\eta] 
 + C\left( \varepsilon^2 m \| \eta\|_{H^1} + \varepsilon m \| \eta\|_{H^1}^2 + m \| \eta\|_{H^1}^3 \right).
 \end{aligned}
\end{equation}
We will now bound the second term in \eqref{eq:lyap_tder}. We observe that
\begin{equation*}
 \begin{aligned}
 - Re\int\ii\eps\d_t(Q_m-2m\d_mQ_m)\bar\eta\,dx
&= - \varepsilon \frac{dm}{dt} \left(\ii \partial_m (Q_m - 2m \partial_m Q_m), \eta \right) \\ 
&= - 2\varepsilon^2 m \left(\ii (\partial_m Q_m + 2 m \partial_{mm} Q_m), \eta \right) \\
& \lesssim 
\varepsilon^2 m \| \eta\|_{L^2} \| \partial_m Q_m\|_{L^2} + \varepsilon^2 m^2 \| \eta\|_{L^2} \| \partial_{mm} Q_m\|_{L^2}.
 \end{aligned}
\end{equation*}
By using $ \| \partial_m Q_m \|_{L^2} \lesssim m^{-\frac{1}{2}}$ and $\| \partial_{mm} Q_m \|_{L^2} \lesssim m^{-\frac{3}{2}}$, see \eqref{eq:dm_Qm} in the Appendix, we have 
\begin{equation}\label{eq:stima2}
 - Re\int\ii\eps\d_t(Q_m-2m\d_mQ_m)\bar\eta\,dx \lesssim \varepsilon^2 \sqrt{m} \| \eta\|_{L^2}.
\end{equation}
By collecting together \eqref{eq:lyap_tder} \eqref{eq:stima1} and \eqref{eq:stima2}, we obtain that there exists $C>0$ such that we can bound the time derivative of $\mathcal{L}_{m,\eps}[\eta]$ as 
\begin{equation}\label{eq:stima3}
 \frac{d}{dt} \mathcal{L}_{m,\eps}[\eta] = -2 \varepsilon \mathcal{L}_{m,\eps}[\eta] + C\left(
 \varepsilon^2 \sqrt{m} \| \eta\|_{H^1} 
 + \varepsilon m \| \eta\|_{H^1}^2 + m \| \eta\|_{H^1}^3 \right).
\end{equation}
Now we want to use Lemma \ref{thm:L_new_eq} to bound $\| \eta\|_{H^1}$. For this purpose, let $\delta^* = \min(\delta_0, \delta_1)>0$, where $\delta_0>0$ is defined in Lemma \ref{thm:L_new_eq}, and let $\varepsilon^*$ be small enough so that $\varepsilon^*/4 < \delta^*$ and $\varepsilon^* \leq \min(\eps_0, \varepsilon_1, \eps_2, \eps_3)$ where $\eps_j$, $j= 0,1,2,3$ are defined in Propositions \ref{thm:impl}, \ref{thm:pos1} and Lemmas \ref{thm:L_new_eq}, \ref{lem:l-phi}, respectively. Suppose that $\eps < \eps^*$ and that the initial condition satisfies 
\begin{equation*}
 \| \eta(0)\|_{H^1} \leq \eps.
\end{equation*}
Then, by continuity, there exists a time interval $[0,t_0]$ such that for any $t\in [0,t_0]$, we have the bound 
\begin{equation}\label{eq:etabou}
 \| \eta(t)\|_{H^1} \leq 2\varepsilon < \delta^* e^{-\varepsilon t}. 
\end{equation}
Thus, in this time interval, we can use inequality \eqref{eq:L_new_eq} and \eqref{eq:etabou} in \eqref{eq:stima3} to obtain that there exists $c>0$ such that 
\begin{equation*}
 \frac{d}{dt} \mathcal{L}_{m,\eps}[\eta] \leq \eps(-2 + c m_0 e^{-2\eps t} ) \mathcal{L}_{m,\eps}[\eta] + \varepsilon^2 c \sqrt{m_0} e^{-\eps t} \sqrt{ \mathcal{L}_{m,\eps}[\eta]}.
\end{equation*}
 Now we make the substitution $N_{m,\varepsilon}[\eta(t)] = e^{\varepsilon t} \sqrt{\mathcal{L}_{m,\varepsilon}[\eta(t)]}$ and get 
		\begin{equation*}
			\begin{split}
				\frac{d}{dt}N_{m,\varepsilon}[\eta(t)] \leq \frac{\eps}{2} c m_0 e^{-2\eps t} N_{m,\varepsilon}[\eta(t)] + \frac{\eps^2}{2} c \sqrt{m_0}. 
			\end{split}
		\end{equation*}
 Integrating the last inequality in time yields
 \begin{equation*}
 N_{m,\varepsilon}[\eta(t)] \leq N_{m,\varepsilon}[\eta(0)] + \frac{\eps^2}{2} c \sqrt{m_0} t + \frac{\eps}{2} c m_0 \int_0^t e^{-2\eps s} N_{m,\varepsilon}[\eta(s)] ds.
 \end{equation*}
		From Gronwall's lemma it follows that there exists $K = K(m(0)) >0$ such that
		\begin{displaymath}
		N_{m,\varepsilon}[\eta(t)] \leq K \left(N_{m,\varepsilon}[\eta(0)] + \varepsilon^2 t\right)
		\end{displaymath}
 which is equivalent to inequality \eqref{eq:L_bound}.
	In particular, for any $t \in [0,t_0]$, we obtain that 
 \begin{equation}\label{eq:tokk2}
 \mathcal L_{m,\varepsilon} [\eta(t)] \leq K (\varepsilon + \eps^4 t_0^2) e^{-2\varepsilon t}.
 \end{equation}
 For $t_0 \leq \eps^{-\frac{3}{2}}$ and any $t \in [0,t_0]$, we thus obtain that $ \| \eta(t)\|_{H^1} \lesssim \varepsilon e^{-\varepsilon t}$. By choosing $\varepsilon^*$ small enough, we obtain that $\| \eta(t_0)\|_{H^1} < \frac{\delta^*}{4}$ and we may repeat the argument above to bootstrap the estimate \eqref{eq:tokk2} to the time interval $[0,t_1]$ where $t_1 > t_0$. Recursively, this proves that estimate \eqref{eq:L_bound} is valid for any $t \leq \eps^{-\frac{3}{2}}$. 
\end{proof}

Inequality \eqref{eq:L_bound} and the closeness of the approximated profile $Q_{m,\varepsilon}$ to the cnoidal profile $Q_m$ \eqref{eq:Q_close} readily imply Theorem \ref{thm:mainCN}. 

\begin{proof}[Proof of Theorem \ref{thm:mainCN}:]
		Let $\psi_0 \in \A$, with $M[\psi(t)] = m(t)$. Let $\gamma(t)$ be defined by \eqref{eq:gamma}. Suppose that the initial condition verifies 
		\begin{equation*}
			\| \eta(0) \|_{H^1} = \| Q_{m(0),\varepsilon} - e^{\ii \gamma(0)}\psi_0\|_{H^1} \leq \varepsilon < \varepsilon^*
		\end{equation*}
		where $Q_{m(0),\varepsilon}$ is defined in Proposition \ref{thm:impl} and $\varepsilon^*>0$ is defined in Proposition \ref{thm:L_bound}. Let $\eta(t,x)$ be the remainder term defined in \eqref{eq:decom_new}. From \eqref{eq:L_new_eq}, \eqref{eq:L_bound}, 
 we get that 
		\begin{equation*}
			\| \eta(t) \|_{H^1}^2 \lesssim 	\mathcal{L}_{m,\varepsilon}[\eta(t)] \lesssim \left(\sqrt{\mathcal{L}_{m,\varepsilon}[\eta(0)]} + \varepsilon^2 t\right)^2 e^{-2\varepsilon t},
		\end{equation*}
 for any $t \leq \eps^{-\frac{3}{2}}$. In particular, if we denote by $t^* = \eps^{-\frac{3}{2}}$, then we obtain that 
 \begin{equation*}
 \| \eta(t^*) \|_{H^1}^2 \lesssim \eps e^{-2\varepsilon t^*} \lesssim \eps.
 \end{equation*}
		 Now let $\xi(t,x)$ be the remainder term defined in \eqref{eq:decomCN}. Then we observe that \eqref{eq:L_new_eq}, \eqref{eq:L_bound} and \eqref{eq:Q_close} imply
		\begin{equation*} 
			\begin{split}
				\| \xi(t^*) \|_{H^1}^2 & \lesssim \| \eta(t^*) \|_{H^1}^2 + \| Q_{m(t^*),\varepsilon} - Q_{m(t^*)} \|_{H^1}^2 \lesssim \eps. 
			\end{split}
		\end{equation*}
 On the other hand, from Lemma \ref{thm:L_equivalence} and Proposition \ref{thm:L_weak}, we obtain that for any $t \geq t^*$, 
 \begin{equation*}
 \| \xi(t) \|_{H^1}^2 \lesssim \mathcal{L}[\psi(t)] \lesssim ( \mathcal{L}[\psi(t^*)] + M[\psi(t^*)]^2 ) e^{-2\varepsilon (t - t^*)}.
 \end{equation*}
 We choose $\eps^*$ to be sufficiently small so that the following inequality is true
 \begin{equation*}
 M[\psi(t^*)] = m_0 e^{-\frac{2}{\sqrt{\eps}}} \leq \sqrt{\eps}.
 \end{equation*}
 This allows us to find that 
 \begin{equation*}
 \| \xi(t) \|_{H^1}^2 \lesssim \eps e^{-2\varepsilon (t - t^*)},
 \end{equation*}
 for any $t \geq t^*$. The last inequality shows that the family of cnoidal profiles $\mathcal F$ is stable for all positive times.
 \end{proof}

\section{Appendix}

\begin{prop}\label{prop:d_om}
 For any $m\in [0,m_0]$, the parameter $\omega$ be defined in \eqref{eq:omegaCN} satisfies 
 \begin{equation}\label{eq:om_pr}
 \frac{d}{dm} \omega(m) \lesssim 1,
 \end{equation}
 and the cnoidal profiles $Q_m$ defined in \eqref{eq:Qm} satisfy
 \begin{equation}\label{eq:dm_Qm}
 \| \partial_m Q_m \|_{L^2} \lesssim m^{-\frac{1}{2}}, \quad \| \partial_{mm} Q_m \|_{L^2} \lesssim m^{-\frac{3}{2}}.
 \end{equation}
\end{prop}

\begin{proof}
We observe that
 \begin{equation*}
 \frac{d \omega(m)}{dm} = \frac{\partial_{k}\omega(k)}{\partial_{k} m(k)}.
 \end{equation*}
To compute the differentiations in $k$, we use the following notations and properties, which can be found in \cite[p. $283$]{ByFr71}:
\begin{equation*}
 K(k) = \int_0^\frac{\pi}2 \frac{1}{\sqrt{1 - k^2\sin^2(\theta)}} d\theta, \quad \mathcal E (k) = \int_0^\frac{\pi}2 \sqrt{1 - k^2\sin^2(\theta)} d\theta,
\end{equation*}
\begin{equation}\label{eq:KEder}
 \partial_{k} K(k) = \frac{\mathcal E (k) + (k^2 - 1) K(k)}{k(1 - k^2)}, \quad \mathcal \partial_{k} \mathcal E(k) = \frac{\mathcal E (k) - K(k)}k,
\end{equation}
\begin{equation}\label{eq:KEder2}
 \partial_{k} (\mathcal E (k) + (k^2 - 1) K(k)) = k K.
\end{equation}
From the equation for the mass $m$ \eqref{eq:mass}, we have
\begin{equation*}
 m = \frac{8}{\pi} K(\mathcal{E} + (k^2 - 1) K) = \frac{8}{\pi} k^2 K\int \frac{\cos^2(\theta)}{\sqrt{1 - k^2\sin^2(\theta)}} d\theta
\end{equation*}
that is 
\begin{equation}\label{eq:mass1}
 \frac{\pi m} {8K} = \mathcal E + (k^2 - 1)K = k^2 \Theta(k) \in \left[0,1\right]
\end{equation}
where 
\begin{equation*}
 \Theta(k) = \int_0^\frac{\pi}2 \frac{cos^2(\theta)}{\sqrt{1 - k^2\sin^2 (\theta)}} d\theta.
\end{equation*}
Now, by using \eqref{eq:KEder} and \eqref{eq:KEder2}, we can see that
\begin{equation}\label{eq:mderK}
\begin{aligned}
 \partial_{k} \left(\frac{\pi}{8}m(k)\right) & = \frac{(\mathcal{E} + (k^2 - 1) K)^2}{k (1 - k^2)} + k K^2 = \frac{k^4 \Theta^2 + k^2(1 - k^2) K^2}{k (1 -k^2)}.
\end{aligned}
\end{equation}
Moreover, we consider $\omega$ defined in \eqref{eq:omegaCN}
\begin{equation*}
 \frac{\pi^2}{4} \omega(k) = K^2 (2k^2 - 1),
\end{equation*}
and we use again \eqref{eq:KEder} and \eqref{eq:KEder2} to obtain
\begin{equation}\label{eq:dk_ome}
\begin{aligned}
 \partial_{k} \left(\frac{\pi^2}{8} \omega(k)\right) &= \frac{ (2k^2 - 1) K (\mathcal{E} + (k^2 - 1) K)}{k(1 - k^2)} + kK^2 \\ 
 &= \frac{k^2(2k^2 - 1) K \Theta + k^2(1 - k^2) K^2}{k(1 - k^2)}. 
\end{aligned}
\end{equation}
Thus from \eqref{eq:mderK} and \eqref{eq:dk_ome}, we obtain that
\begin{equation}\label{eq:omegaonM}
 \begin{aligned}
 \pi \frac{ \partial_k \omega}{\partial_k m} = \frac{(2k^2 - 1) K \Theta + (1 - k^2) K^2}{k^2 \Theta^2 + (1 - k^2) K^2}.
 \end{aligned}
\end{equation}
Now we observe that 
\begin{equation}\label{eq:thKto0}
 \lim_{k \to 0} \Theta (k) = \frac{\pi}{4}, \quad \lim_{k \to 0} K (k) = \frac{\pi}{2}, \quad \lim_{k \to 0} \mathcal{E} (k) = \frac{\pi}{2}
\end{equation}
which implies
\begin{equation*}
 \lim_{k \to 0} \left((2k^2 - 1) K \Theta + (1 - k^2) K^2\right) = \frac{\pi^2}{8}
\end{equation*}
and 
\begin{equation*}
 \lim_{k \to 0}\left(k^2 \Theta^2 + (1 - k^2) K^2 \right)= \frac{\pi^2}{4}.
\end{equation*}
Thus we obtain that
\begin{equation}\label{eq:kto0}
 \lim_{k \to 0} \frac{ \partial_k \omega(k)}{\partial_k m(k)} = \frac{2}{\pi}.
\end{equation}
On the other hand, for $k$ far from zero, we notice that there exists a constant $C>0$ such that
\begin{equation*}
 k^2 \Theta^2 + (1 - k^2) K^2 > C,
\end{equation*}
and 
\begin{equation*}
 (2k^2 - 1) K \Theta + (1 - k^2) K^2 \leq C m \leq C m_0.
\end{equation*}
In particular, from \eqref{eq:kto0} and from the continuity of $\partial_k \omega(k)$ and $\partial_k m(k)$ with respect to $k$, we obtain that there exists a constant $C = C(m_0)$ such that
\begin{equation*}
 \frac{\partial_k \omega(k)}{\partial_k m(k)} \leq C.
\end{equation*}
It remains to prove \eqref{eq:dm_Qm}. We will now compute the derivative of the cnoidal profile with respect to the mass. We observe that
\begin{equation*}
 \partial_m Q_m = \frac{\partial_k Q_k}{ \partial_k m(k)},
\end{equation*}
and 
\begin{equation} \label{eq:dmm_Qm}
 \partial_{mm} Q_m = (\partial_{k} m(k))^{-2}\left(
 \partial_{kk} Q_k - \frac{\partial_{kk} m(k)}{\partial_{k} m(k)} \partial_{k} Q_k \right),
\end{equation}
where 
\begin{equation*}
 Q_k = \sqrt{2} k\frac{ 2 K(k)}{\pi} cn \left(\frac{2 K(k)}{\pi} x,k \right).
\end{equation*}
We recall that 
\begin{equation}\label{eq:cn_der}
 \partial_kcn(x,k) = \frac{sn(x,k) dn(x,k)}{k (1 - k^2)} \left( (k^2 - 1) x + \mathcal E(k) - k^2 \frac{sn(x,k) cn(x,k)}{dn(x,k)}\right) 
\end{equation}
see \cite[Formula 710.52]{ByFr71}). Thus, by using \eqref{eq:KEder} and \eqref{eq:cn_der}, we compute the differentiation of $Q_k$ with respect to $k$ and obtain that
\begin{equation}\label{eq:QderkLong}
 \begin{aligned}
 \frac{1}{\sqrt{2}} \partial_k Q_k &= \frac{2K}{\pi} cn(y,k) + \frac{2}{\pi} \frac{k^2 \Theta}{(1 - k^2)} cn(y,k) + \frac{2K}{\pi} \frac{2}{\pi} \frac{k^2 \Theta}{(1 - k^2)} \partial_y cn(y,k) \\
 & + \frac{2K}{\pi}\frac{sn(y,k) dn(y,k)}{ (1 - k^2)} \left( (k^2 - 1) y + \mathcal E(k) - k^2 \frac{sn(y,k) cn(y,k)}{dn(y,k)}\right) 
 \end{aligned}
\end{equation}
where 
\begin{equation*}
 y = \frac{2K(k)}{\pi} x.
\end{equation*}
In particular, since
\begin{equation*}
\begin{aligned}
 & \lim_{k\to 0} cn(y,k) = \cos(x), \quad \lim_{k\to 0} sn(y,k) = \sin(x), \quad \lim_{k\to 0} dn(y,k) = \lim_{k\to 0}\frac{2K}{\pi} = 1,
\end{aligned}
\end{equation*}
we have that
\begin{equation*}
 \lim_{k \to 0} \frac{1}{\sqrt{2}} \partial_k Q_k(x) = \cos(x) + \sin(x) \left(\frac{\pi}{2} - x\right),
\end{equation*}
where we used \eqref{eq:thKto0}. We can easily compute that 
\begin{equation*}
 \int_0^{2\pi} \left|\cos(x) + \sin(x) \left(\frac{\pi}{2} - x\right)\right|^2 dx = \frac{7 \pi^3 + 6\pi}{12}. 
\end{equation*}
Thus, after cumbersome but straightforward computations, we can see that 
\begin{equation*}
 \lim_{k \to 0} \| \partial_k Q_k\|_{L^2}^2 = \frac{7 \pi^3 + 6\pi}{6}.
\end{equation*}
As a consequence, equation \eqref{eq:mderK} implies that
\begin{equation*}
 \lim_{k \to 0} \frac{ \| \partial_k Q_k\|_{L^2}}{\partial_k m(k)} \lesssim \lim_{k \to 0} \frac{k (1 - k^2)}{k^4 \Theta^2 + k^2(1 - k^2) K^2} \lesssim \lim_{k\to 0} \frac{k}{k^2} \lesssim \lim_{k \to 0} \frac{1}{k} = \infty.
\end{equation*}
Consequently, close to $k= 0$, this fraction behaves like $k^{-1}$. On the other hand, from \eqref{eq:mass1} and \eqref{eq:thKto0}, we see that $m \sim k^2$ when $k$ is close to zero. Thus we can conclude that 
\begin{equation}\label{eq:dk_QoverM}
 \frac{ \| \partial_k Q_k\|_{L^2}}{\partial_k m(k)} \lesssim \frac{1}{\sqrt{m}}
\end{equation}
for $k$ close to zero. On the other hand, for $k$ far from zero, this fraction is always bounded by a constant depending on $m_0$. This implies the first bound in \eqref{eq:dm_Qm} for any $m \in [0,m_0]$.
\newline
Finally, it remains to prove the second bound in \eqref{eq:dm_Qm}. By using \eqref{eq:mderK}, \eqref{eq:KEder}, \eqref{eq:KEder2} and \eqref{eq:mass1}, we see that
\begin{equation*}
 \begin{aligned}
 \frac{\pi}{8}\partial_{kk} m(k) & = \frac{2(\mathcal{E} + (k^2 -1)K^2) kK}{k(1 - k^2)} - \frac{(\mathcal{E} + (k^2 -1)K^2)^2(1 - 3k^2)}{k^2(1 - k^2)^2} \\
 & + K^2 + \frac{2 k K (\mathcal{E} + (k^2 -1)K^2) }{k(1 - k^2)} \\
 &= \frac{2 k^2 \Theta K}{(1 - k^2)} - \frac{k^2 \Theta^2 (1 - 3k^2)}{(1 - k^2)^2} + K^2 + \frac{2k^2 K\Theta}{(1 - k^2)^2}.
 \end{aligned}
\end{equation*}
In particular, since $K(k) \to \frac{\pi}{2}$ as $k \to 0$, we have that 
\begin{equation}\label{eq:dkk_m}
 \lim_{k \to 0}\partial_{kk} m(k) = 2\pi.
\end{equation}
Lastly, we shall compute $\partial_{kk} Q_k$. It is a very long and technical computation, but using \eqref{eq:QderkLong} and several Formulas in \cite[p. $283$]{ByFr71}, one can see that 
\begin{equation*}
 \| \partial_{kk} Q_k\|_{L^2} \lesssim \frac{1}{k}
\end{equation*}
for $k$ close to zero. Using again the fact that $m \sim k^2$ for $k$ close to zero, we obtain from \eqref{eq:dmm_Qm}, \eqref{eq:dkk_m} and \eqref{eq:dk_QoverM} that 
\begin{equation*}
 \| \partial_{mm} Q_m \|_{L^2} \lesssim \frac{\| \partial_{kk} Q_k\|_{L^2}}{(\partial_k m(k))^2} + \frac{\| \partial_k Q_k\|_{L^2}}{\partial_k m(k)} \frac{\partial_{kk} m(k)}{(\partial_k m(k))^2} \lesssim \frac{1}{k^3} \sim \frac{1}{m^\frac{3}{2}},
\end{equation*}
for $k$ close to zero. Again, since this fraction is bounded for $k$ far from zero, this implies the second inequality in \eqref{eq:dm_Qm} for any $m \in [0,m_0]$. 
\end{proof}

\bibliographystyle{plain}
\bibliography{bibliomain}  

\end{document}